\documentclass[11pt,twoside]{amsart}
\usepackage[utf8]{inputenc}
\usepackage[T1]{fontenc}
\usepackage[english]{babel}
\usepackage{amsmath, amssymb, amsthm}            
\usepackage{amstext, amsfonts, a4}
\usepackage{graphicx}
\usepackage{pifont}
\usepackage{hyperref}
\usepackage{color}
\usepackage{stmaryrd}
\usepackage{enumitem}
\usepackage{tikz,tkz-euclide}
\usetikzlibrary{calc}
\usetikzlibrary{arrows}
\tikzset{hidden/.style = {thick, dashed}}
\usetikzlibrary{math}
\usetikzlibrary{shapes.geometric}

\usepackage{pgfplots}
\pgfplotsset{compat=1.15}
\usepackage{mathrsfs}

\theoremstyle{plain}
	\newtheorem{Theo}{Theorem}[section] 
	\newtheorem{Prop}[Theo]{Proposition}        
	\newtheorem{Lem}[Theo]{Lemma}            
	\newtheorem{Cor}[Theo]{Corollary}

\theoremstyle{definition}
	\newtheorem{Def}[Theo]{Definition}
        
	\newtheorem{Nota}[Theo]{Notation}

\theoremstyle{remark}
	\newtheorem{Rema}[Theo]{Remark}

\def\RR{{\mathbb R}}    
\def\HH{{\mathbb H}^2}    
\newcommand{\Hyp}{{\mathbb H}}
\newcommand{\KVol}{\mbox{KVol}}

\newcommand{\Int}{\mbox{Int}}
\newcommand{\Vol}{\mbox{Vol}}

\makeatletter

\def\hyper@x#1,#2\relax{#1}
\def\hyper@y#1,#2\relax{#2}
\def\hyper@coords#1{#1}

\newif\ifhyper@vertical

\def\hyper@computer#1#2{%
  \edef\hyper@toscan{(#1)}
  \tikz@scan@one@point\hyper@coords\hyper@toscan
  \edef\hyper@sx{\the\pgf@x}
  \edef\hyper@sy{\the\pgf@y}
  \edef\hyper@toscan{(#2)}
  \tikz@scan@one@point\hyper@coords\hyper@toscan
  \edef\hyper@ex{\the\pgf@x}
  \edef\hyper@ey{\the\pgf@y}
  \pgfmathsetmacro{\hyper@mx}{(\hyper@ex + \hyper@sx)/2}
  \pgfmathsetmacro{\hyper@my}{(\hyper@ey + \hyper@sy)/2}
  \pgfmathsetmacro{\hyper@dx}{\hyper@ex - \hyper@sx}
  \pgfmathparse{\hyper@dx == 0 ? "\noexpand\hyper@verticaltrue" : "\noexpand\hyper@verticalfalse"}
  \pgfmathresult
  \ifhyper@vertical
  \edef\hyper@cmd{-- (\tikztotarget)}
  \else
  \pgfmathsetmacro{\hyper@dy}{\hyper@ey - \hyper@sy}
  \pgfmathsetmacro{\hyper@t}{\hyper@my/\hyper@dx}
  \pgfmathsetmacro{\hyper@cx}{\hyper@mx + \hyper@t * \hyper@dy}
  \pgfmathsetmacro{\hyper@radius}{veclen(\hyper@cx - \hyper@sx, \hyper@sy)}
  \pgfmathsetmacro{\hyper@sangle}{180 - atan2(\hyper@sy,\hyper@cx-\hyper@sx)}
  \pgfmathsetmacro{\hyper@eangle}{180 - atan2(\hyper@ey,\hyper@cx-\hyper@ex)}
  \edef\hyper@cmd{arc[radius=\hyper@radius pt, start angle=\hyper@sangle, end angle=\hyper@eangle]}
  \fi
}

\def\hyper@disc@computer#1#2{%
  \edef\hyper@toscan{(#1)}
  \tikz@scan@one@point\hyper@coords\hyper@toscan
  \edef\hyper@sx{\the\pgf@x}
  \edef\hyper@sy{\the\pgf@y}
  \edef\hyper@toscan{(#2)}
  \tikz@scan@one@point\hyper@coords\hyper@toscan
  \edef\hyper@ex{\the\pgf@x}
  \edef\hyper@ey{\the\pgf@y}
  \pgfmathsetmacro{\hyper@det}{\hyper@sx * \hyper@ey - \hyper@sy * \hyper@ex}
  \pgfmathparse{\hyper@det == 0 ? "\noexpand\hyper@verticaltrue" : "\noexpand\hyper@verticalfalse"}
  \pgfmathresult
  \ifhyper@vertical
  \edef\hyper@cmd{-- (\tikztotarget)}
  \else
  \pgfmathsetmacro{\hyper@mx}{(\hyper@ex + \hyper@sx)/2}
  \pgfmathsetmacro{\hyper@my}{(\hyper@ey + \hyper@sy)/2}
  \pgfmathsetmacro{\hyper@dx}{\hyper@ex - \hyper@sx}
  \pgfmathsetmacro{\hyper@dy}{\hyper@ey - \hyper@sy}
  \pgfmathsetmacro{\hyper@dradius}{\pgfkeysvalueof{/tikz/hyperbolic disc radius}}
  \pgfmathsetmacro{\hyper@t}{((\hyper@dradius)^2 - \hyper@sx * \hyper@ex - \hyper@sy * \hyper@ey)/(2 * (\hyper@sx * \hyper@ey - \hyper@sy * \hyper@ex))}
  \pgfmathsetmacro{\hyper@radius}{sqrt((\hyper@t)^2 + .25) * veclen(\hyper@dx,\hyper@dy)}
  \pgfmathsetmacro{\hyper@cx}{\hyper@mx + \hyper@t * \hyper@dy}
  \pgfmathsetmacro{\hyper@cy}{\hyper@my - \hyper@t * \hyper@dx}
  \pgfmathsetmacro{\hyper@sangle}{atan2(\hyper@sy-\hyper@cy,\hyper@sx - \hyper@cx)}
  \pgfmathsetmacro{\hyper@eangle}{atan2(\hyper@ey-\hyper@cy,\hyper@ex - \hyper@cx)}
  \pgfmathsetmacro{\hyper@eangle}{\hyper@eangle > \hyper@sangle + 180 ? \hyper@eangle - 360 : \hyper@eangle}
  \edef\hyper@cmd{arc[radius=\hyper@radius pt, start angle=\hyper@sangle, end angle=\hyper@eangle]}
\fi
}

\def\hyper@plane@tangent#1#2{%
  \edef\hyper@toscan{(#1)}
  \tikz@scan@one@point\hyper@coords\hyper@toscan
  \edef\hyper@sx{\the\pgf@x}
  \edef\hyper@sy{\the\pgf@y}
  \edef\hyper@toscan{(#2)}
  \tikz@scan@one@point\hyper@coords\hyper@toscan
  \edef\hyper@ex{\the\pgf@x}
  \edef\hyper@ey{\the\pgf@y}
  \pgfmathsetmacro{\hyper@ex}{\hyper@ex - \hyper@sx}
  \pgfmathsetmacro{\hyper@ey}{\hyper@ey - \hyper@sy}
  \pgfmathparse{\hyper@ex == 0 ? "\noexpand\hyper@verticaltrue" : "\noexpand\hyper@verticalfalse"}
  \pgfmathresult
  \ifhyper@vertical
  \pgfmathsetmacro{\hyper@d}{\hyper@ey/1cm}
  \pgfmathsetmacro{\hyper@radius}{\hyper@sy * exp(\hyper@d) - \hyper@sy}
  \edef\hyper@cmd{-- ++(0,\hyper@radius pt)}
  \else
  \pgfmathsetmacro{\hyper@d}{\hyper@ex > 0 ? veclen(\hyper@ex,\hyper@ey) : -veclen(\hyper@ex,\hyper@ey)}
  \pgfmathsetmacro{\hyper@radius}{abs(\hyper@sy * \hyper@d / \hyper@ex)}
  \pgfmathsetmacro{\hyper@sangle}{90 + atan(\hyper@ey/\hyper@ex)}
  \pgfkeysgetvalue{/tikz/hyperbolic plane target angle}{\hyper@eangle}
  \ifx\hyper@eangle\pgfutil@empty
  \pgfmathsetmacro{\hyper@d}{\hyper@d/1cm}
  \pgfmathsetmacro{\hyper@ey}{\hyper@ey/1cm}
  \pgfmathsetmacro{\hyper@tanhd}{tanh(\hyper@d)}
  \pgfmathsetmacro{\hyper@eangle}{acos((\hyper@d * \hyper@tanhd - \hyper@ey)/(\hyper@d - \hyper@ey * \hyper@tanhd))}
  \fi
  \edef\hyper@cmd{arc[radius=\hyper@radius pt, start angle=\hyper@sangle, end angle=\hyper@eangle]}
\fi
}

\tikzset{%
  hyperbolic disc radius/.initial={1cm},
  hyperbolic plane/.style={
    to path={
      \pgfextra{\hyper@computer\tikztostart\tikztotarget}
      \hyper@cmd
    }
  },
  hyperbolic plane tangent/.style={
    to path={
      \pgfextra{\hyper@plane@tangent\tikztostart\tikztotarget}
      \hyper@cmd
    }
  },
  hyperbolic disc/.style={
    to path={
      \pgfextra{\hyper@disc@computer\tikztostart\tikztotarget}
      \hyper@cmd
    }
  },
  hyperbolic plane target angle/.initial={},
}

\makeatother

\title{Algebraic intersection for a family of Veech surfaces}
\begin{document}

\author{Julien Boulanger}
\address{
UMR CNRS 5582,
Univ. Grenoble Alpes, CNRS, Institut Fourier, F-38000 Grenoble, France}
\email{julien.boulanger@univ-grenoble-alpes.fr}

\author{Erwan Lanneau}
\address{
UMR CNRS 5582,
Univ. Grenoble Alpes, CNRS, Institut Fourier, F-38000 Grenoble, France}
\email{erwan.lanneau@univ-grenoble-alpes.fr}

\author{Daniel Massart}
\address{
IMAG, CNRS,  Univ  Montpellier, France
}
\email{daniel.massart@umontpellier.fr }


\maketitle

%

\section{Introduction}
Our objects of study are translation surfaces and their geodesics. These structures arise in the study of rational polygonal billiards and more generally in Teichm\"uller dynamics. This paper focuses on  computing relations between lengths of geodesics and their intersections on those surfaces. 

\subsection{Motivation and results}
For any closed (meaning compact, connected, without boundary) oriented surface $X$, the algebraic intersection endows the first homology group $H_1(X,\RR)$ with a symplectic bilinear form denoted $\mbox{Int} (\cdot, \cdot)$. When $X$ is endowed with a Riemannian metric $g$ (possibly with singularities), one can ask the following question: how much can two curves of a given length intersect? Namely, what~is
\begin{equation}
\label{eq:KVol}
\KVol(X): = \Vol(X,g)\cdot \sup_{\alpha,\beta} \frac{\Int (\alpha,\beta)}{l_g (\alpha) l_g (\beta)},
\end{equation}
where the supremum ranges over all piecewise smooth closed curves $\alpha$ and $\beta$ in $X$, and $l_g(\cdot)$ denotes the length with respect to the Riemannian metric. It is readily seen that multiplying by the volume $\Vol (X,g)$ makes the quantity invariant by rescaling the metric $g$. This function is well defined, finite (see~\cite{MM}) and continuous in the metric.

Recent work~\cite{CKM,CKMcras} provides new estimates of $\KVol$ on the Teichm\"uller curve of some arithmetic translation surfaces $(X,\omega)$. In this paper we study the most natural family of non-arithmetic Teichm\"uller curves $(\mathcal T_n)_{n \geq 5}$, generated for $n \geq 5$ by the original Veech surface described in~\cite{Veech}, namely the surface arising from the unfolding of a right-angled triangle with angles $(\pi/2,\pi/n,(n-2)\cdot \pi/2n)$.
For odd $n$, $\mathcal T_n$ is canonically identified with $\mathbb H^2/\Gamma_n$ where $\Gamma_n$ is the Hecke triangle group of signature $(2,n,\infty)$ (see~\S\ref{sec:prem} and~\S\ref{sec:SL2R}). In this context, we establish the first known explicit formula for $\KVol$ (beyond the case of the moduli space of flat tori, see~\cite{MM}):
\begin{Theo}\label{thm:main:intro}
For odd integer $n \geq 5$, and any $(X,\omega) \in \mathcal T_n$, we have
\[ 
\KVol(X,\omega)=\frac{\frac{n}{2}  \cot \frac{\pi}{n} \cdot \frac1{\sin \frac{\pi}{n}}}{\cosh  d_{\mathrm{hyp}} (X,\gamma_{0,\infty})},
 \]
where 
$\gamma_{0,\infty}$ is the hyperbolic geodesic in $\Hyp^2$ with endpoints $0$ and $\infty$.
\end{Theo}
In particular $\KVol$ is real analytic on $\mathcal T_n$ except along the  geodesic with endpoints $\cos \frac{\pi}{n}$ and  $\infty$.
\subsection{Context and history}
While $\KVol$ is a close cousin of the systolic ratio $\sup_{\alpha} \frac{\Vol(X,g)}{l_g (\alpha)^2}$, very little is known on the function $\KVol$. For any Riemannian surface $(X,g)$, we have $\KVol(X,g) \geq 1$, and equality holds if and only if 
$(X,g)$ is a flat torus~\cite{MM}. Almost all of the obvious questions about $\KVol$ on hyperbolic surfaces are currently open.
In this paper we propose to continue the study of $\KVol$ as a function on the moduli space of translation surfaces, originally initiated by the third named author in~\cite{CKM,CKMcras}.  The surface $X$ is called a translation surface if it is equipped with a translation structure, that is, an atlas of charts such that all transition functions are translations in $\mathbb R^2$. We assume that the chart domains cover all the surface $X$ except for finitely many points called singularities. The translation structure induces the structure of a smooth manifold and a flat Riemannian metric on $X$ punctured at the singular points. We require that the metric have a cone type singularity at each singular point. Although geodesics on $X$ are piecewise straight lines, they can be quite complicated: they may be unions of saddle connections with different directions.

More specifically, in~\cite{CKM}, $\mbox{KVol}$ is studied for ramified covers of the torus (or arithmetic Teichm\"uller curves). It is proved in~\cite{CKM} that $\mbox{KVol}$, defined on the Teichm\"uller disc of the surface tiled with three squares, is unbounded, but it does have a minimum, achieved at a surface, unique modulo symmetries, and otherwise fairly undistinguished. The interesting surfaces, i.e. the three square surfaces, and the surface tiled with six equilateral triangles, are local maxima, with $\mbox{KVol}=3$, where $3$ should be thought of as the ratio of the total area of the surface, to the area of the smallest cylinder of closed geodesics. The local maxima are not locally unique, they come in hyperbolic geodesics, in the Teichm\"uller curve viewed as a quotient of the hyperbolic plane by a Fuchsian group.

In the current paper, we extend the study to  non arithmetic Teichm\"uller curves.

\subsection{$\KVol$ on non arithmetic Teichm\"uller curves}
Teichm\"uller curves are isometrically immersed algebraic curves in the moduli space of Riemann surfaces. These arise as $\mathrm{SL}(2,\RR)$-orbit (or Teichm\"uller disc) of special flat surfaces that are called Veech surfaces. For $n=2m+1\geq 5$ we will denote by $X_0$ the original Veech surface~\cite{Veech} arising from the unfolding a right-angled triangle with angles $(\pi/2,\pi/n,(n-2)\cdot \pi/2n)$, also referred to as the \emph{double regular $n-$gon}. The surface $X_0$ lies in the stratum $\mathcal H(n-3)$ and its genus is $\mathrm{genus}(X_0)=(n-1)/2$. From Theorem \ref{thm:main:intro}, we deduce the
\begin{Cor}\label{cor:main}
	For any $X\in \mathcal T_n$ the following holds
$$
\frac{n}{2}  \cot \frac{\pi}{n} \leq \KVol(X) \leq \frac{n}{2}  \cot \frac{\pi}{n} \cdot \frac1{\sin \frac{\pi}{n}}.
$$
Moreover the bounds are sharp and:
\begin{enumerate}
\item The maximum of the function $\KVol$ on $\mathcal T_n$ is achieved, precisely, along $\gamma_{0,\infty}$, that is, by  images of the right-angled staircases under the Teichm\"uller geodesic flow (see Figure \ref{heptagon}).
\item The minimum of the function $\KVol$ on $\mathcal T_n$ is achieved, uniquely,  at $X_0$. 
\end{enumerate}
Finally, in the definition of $\KVol$, the supremum is achieved by pairs of curves that are (images of) pairs of sides of the double regular $n-$gon.
\end{Cor}

\begin{Rema}
The case of the regular $4m$-gon is dealt with in a forthcoming paper of the first author. When $n\equiv 2 \mod 4$, the surface $X_n$ belongs to a stratum with two conical points. This case is combinatorially a bit more complicated and we set it  aside  for future work.
\end{Rema}

Unlike the three-square surface case~\cite{CKM}, $\KVol$ is bounded on  the Teichm\"uller discs of  the double regular $n-$gon. More generally we show
\begin{Theo}
\label{thm:bound}
The function $\KVol$ is bounded on the Teichm\"uller disc of $X\in \mathcal H(2)$ if and only if $X$ is a primitive Veech surface i.e. $X$ is not arithmetic.
\end{Theo}
The same discussion applies to the Teichm\"uller disc of surfaces in the Prym eigenform loci in $\mathcal H(4)$ and $\mathcal H(6)$ (see \cite{Lanneau_Nguyen_Prym}) by looking at possible cylinder diagrams. We suspect that $\KVol$ is never bounded from above on the Teichm\"uller disc of arithmetic Veech surfaces. The first author recently obtained an example of a non arithmetic Teichm\"uller disc where $\KVol$ is {\em unbounded}.
\begin{Rema}
Another difference with~\cite{CKM} is that the minimum of $\KVol$ on $\mathcal T_n$ is achieved, uniquely, by the most interesting surface in the disc, namely the double $n-$gon. On the other hand, similarly to the three-square case, the local maxima, which are also global maxima are achieved along hyperbolic geodesics in the Teichm\"uller disc, which correspond to surfaces with a right-angled template, see Figure~\ref{heptagon}.
\end{Rema}

\subsection{Comparison of norms}
The quantity $\KVol$ appears naturally in the comparisons between the stable norm $||\cdot||_s$ and the Hodge norm. More precisely (see~\cite{Massart1996,MM}):
$$
||\cdot||_s \leq \sqrt{\mathrm{Vol(M,g)}} ||\cdot||_{\mathrm Hodge} \leq \KVol(M,g) ||\cdot||_s
$$
The upper bound is sharp. More recently, in the context of quadratic differentials, a new comparison between Hodge and Teichm\"uller norms has been established in~\cite{KW2022}. The latter is useful for the proof of effective versions of results on dynamics of mapping class groups (see the work~\cite{AH}). It would be interesting to establish an analogue of Theorem~\ref{thm:main:intro} for the Hodge and Teichm\"uller norms and see wether it gives better effective results on mapping class groups.

\subsection{Organization of the paper}
In \S\ref{sec:prem} we recall prerequisites on translation surfaces, Veech groups, Veech surfaces, and we describe the right-angled staircase models for our surfaces. In \S\ref{section:boundedness} we explain why $\KVol$ is bounded on some Teichmüller curves. In \S\ref{sec:SL2R} and \S\ref{sec:another look} we explain how to interpret $\KVol$ geometrically, in terms of hyperbolic distance on the Teichmüller curve (identified with a quotient of $\mathbb H^2$ by a Fuchsian group), and we give an upper bound for KVol on the  Teichmüller curve. In particular this proves that the maximum is achieved by the staircase surfaces. 
 
In \S\ref{sec:2m+1-gon} we perform the first main step of the proof: we compute $\KVol$ for the double $n$-gon, by a geometric method, carefully looking at how saddle connections intersect depending on their directions.
 
In \S\ref{sec:extension} we perform the second main step of the proof: we interpolate, by analytical methods, between the double $n$-gon and the staircase surfaces, thus proving that the minimum is achieved, uniquely, by the former. 
 
\subsection*{Acknowledgments} This work has been partially supported by the LabEx PERSYVAL-Lab (ANR-11-LABX-0025-01) funded by the French program Investissement d’avenir.

\section{Background}
\subsection{Preliminaries}
\label{sec:prem}

We review useful results concerning flat surfaces, Veech groups and Teichm\"uller curves. For a general introduction to translation surfaces and their moduli spaces,
we refer the reader to the surveys~\cite{survey_Zorich, survey_Wright, survey_Massart} and the references therein.
Throughout this paper, $\mathcal H(k_1,\dots,k_r)$ will denote the stratum of moduli space of Abelian differentials $\omega$ on Riemann surfaces $X$ of genus $g\geq 1$, having $r$ zeros, of orders $k_i$ for $i=1,\dots,r$, or equivalently of angles $2\pi(k_i+1)$. For simplicity we will simply denote $X$ for the pair $(X,\omega)$ if the context is clear.

As previously mentioned, geodesics on translation surfaces can be complicated. However for $X$ in the minimal stratum $\mathcal H(2g-2)$, any geodesic is homologous to a union of {\em closed} saddle connections. Hence the supremum in Equation~\eqref{eq:KVol} can be taken over all saddle connections of the surface.

We denote the holonomy vector of a saddle connection $\alpha$ by $\overrightarrow{\alpha} = \int_\alpha \omega \in \RR^2$. By abuse of notation we will often confuse $\overrightarrow{\alpha}$ with $\alpha$. We have $l(\alpha)=||\overrightarrow{\alpha}||$.

A Teichm\"uller curve is an isometrically immersed complex curve in the moduli space of Riemann surfaces. The projection of a closed $\mathrm{GL}_2(\RR)$-orbit to the moduli space of Riemann surfaces is a Teichm\"uller curve, so as is common we will also refer to closed $\mathrm{GL}_2(\RR)$-orbits as Teichm\"uller curves in $\mathcal H(k_1,\dots,k_r)$. A point on a closed $\mathrm{GL}_2(\RR)$-orbit is called a Veech surface. Equivalently a translation surface is a Veech surface if and only if its Veech group is a lattice in $\mathrm{SL}_2(\RR)$. The easiest examples of Teichm\"uller curves are orbits of arithmetic surfaces (namely ramified covers of the two-torus). Veech~\cite{Veech} discovered a family of non arithmetic surfaces arising from a right-angled triangle with angles $(\pi/2,\pi/n,(n-2)\cdot \pi/2n)$ by the unfolding construction described in~\cite{KZ}. We give a description of these surfaces in the next section.

\subsection{The staircase model for the double $n-$gon}
\label{sec:even:odd}

For $n\geq 3$, the surface $X_0$ (arising from the unfolding of a right triangle with angles $(\pi/2,\pi/n,(n-2)\pi/2n)$) can be described as follows (see~\cite{KZ,Veech})
\begin{itemize}
\item If $n\geq 8$ is even then $X_0$ is the quotient of the regular $n-$gon (with radius $1$) by gluing opposite sides by translation. Moreover, if $n = 4m$ for $m \geq 2$ the surface belongs to the stratum $\mathcal H(2m-2)$, while if $n = 4m+2$ it belongs to the stratum $\mathcal H(m-1,m-1)$.
\item If $n\geq 5$ is odd, then $X_0$ is the quotient of the double of the regular $n-$gon (with radius $1$) by gluing opposite sides by translation. It belongs to the stratum $\mathcal H(n-3)$.
\end{itemize}
As mentioned in the introduction we will focus on the latter case and in the sequel we will assume $n=2m+1$ is an odd integer. 

The double $n$-gon $X_0$ has a staircase model $S_0$ in its $\mathrm{GL}_2(\RR)$-orbit, drawn in Figure~\ref{heptagon} and described in~\cite{Monteil}, which we shall use  with a few modifications (see also~\cite[\S 5]{Veech}).
\begin{figure} [h!]
\begin{center}
\hskip -7mm
\begin{tikzpicture}[line cap=round,line join=round,>=triangle 45,x=1cm,y=1cm, scale=.5,rotate=90]
\draw (-1,0.1) node[above]{${\scriptstyle 2}$};
\draw (-0.6,1.5) node[above]{${\scriptstyle 0}$};
\draw (0.5,0.7) node[above]{${\scriptstyle 1}$};
\draw (0.5,-0.8) node[above]{${\scriptstyle 3}$};
\draw (-1,-1.5) node[above]{${\scriptstyle 4}$};
\draw (0,-2.1) node[above]{${\scriptstyle 5}$};
\draw (-1,-2.6) node[above]{${\scriptstyle 6}$};
\draw (0,-3.7) node[above]{${\scriptstyle 7}$};
\draw (-1,-4.3) node[above]{${\scriptstyle 8}$};
\draw (-0.6,-5.2) node[above]{${\scriptstyle 9}$};

\draw [line width=1pt] (0,2)-- (-1.5636629649360594,1.2469796037174672)-- (-1.9498558243636472,-0.44504186791262856)-- (-0.8677674782351166,-1.801937735804838)-- (0.8677674782351164,-1.8019377358048383)-- (1.9498558243636472,-0.4450418679126288)-- (1.5636629649360596,1.246979603717467);
\draw [line width=1pt] (1.5636629649360596,1.246979603717467)-- (0,2);
\draw [line width=1pt] (0,-5.603875471609676)-- (-1.56366296493606,-4.850855075327143)-- (-1.9498558243636475,-3.158833603697047)-- (-0.8677674782351166,-1.8019377358048378)-- (0.8677674782351164,-1.801937735804838)-- (1.949855824363647,-3.158833603697048)-- (1.563662964936059,-4.850855075327143);
\draw [line width=1pt] (1.563662964936059,-4.850855075327143)-- (0,-5.603875471609676);
\draw [line width=.8pt,dash pattern=on 1pt off 1pt] (1.563662964936059,-4.850855075327143)-- (-1.9498558243636475,-3.158833603697047);
\draw [line width=.8pt,dash pattern=on 1pt off 1pt] (1.949855824363647,-3.158833603697048)-- (-0.8677674782351166,-1.801937735804838);
\draw [line width=.8pt,dash pattern=on 1pt off 1pt] (0.8677674782351164,-1.8019377358048383)-- (-1.9498558243636472,-0.44504186791262856);
\draw [line width=.8pt,dash pattern=on 1pt off 1pt] (1.9498558243636472,-0.4450418679126288)-- (-1.5636629649360594,1.2469796037174672);
\draw [line width=.8pt,dash pattern=on 1pt off 1pt] (1.563662964936059,-4.850855075327143)-- (-1.56366296493606,-4.850855075327143);
\draw [line width=.8pt,dash pattern=on 1pt off 1pt] (1.949855824363647,-3.158833603697048)-- (-1.9498558243636475,-3.158833603697047);
\draw [line width=.8pt,dash pattern=on 1pt off 1pt] (1.9498558243636472,-0.4450418679126288)-- (-1.9498558243636472,-0.44504186791262856);
\draw [line width=.8pt,dash pattern=on 1pt off 1pt] (1.5636629649360596,1.246979603717467)-- (-1.5636629649360594,1.2469796037174672);
\end{tikzpicture}
%
\begin{tikzpicture}[line cap=round,line join=round,>=triangle 45,x=1cm,y=1cm, scale=.5]
\draw (2,15) node[above]{cut \& paste};
\draw (2,15) node[below]{rotate by ${\scriptstyle 3\pi/2}$};
\draw [->,>=stealth] (0,15) -- (4,15);
\end{tikzpicture}
\begin{tikzpicture}[line cap=round,line join=round,>=triangle 45,x=1cm,y=1cm, scale=.5]
\draw (6.5,-11.8) node[above]{${\scriptstyle 0}$};
\draw (5,-11.6) node[above]{${\scriptstyle 9}$};
\draw (4.4,-10.8) node[above]{${\scriptstyle 8}$};
\draw (2.5,-10.3) node[above]{${\scriptstyle 1}$};
\draw (-6.8,-8.7) node[above]{${\scriptstyle 5}$};
\draw (-5.6,-9.1) node[above]{${\scriptstyle 4}$};
\draw (-4.2,-8.7) node[above]{${\scriptstyle 3}$};
\draw (-2.8,-9.1) node[above]{${\scriptstyle 6}$};
\draw (-1.6,-10.3) node[above]{${\scriptstyle 7}$};
\draw (0.5,-10.4) node[above]{${\scriptstyle 2}$};

\draw [line width=1pt] (5,-11.603875471609676)-- (8.127325929872118,-11.603875471609676);
\draw [line width=.8pt,dash pattern=on 1pt off 1pt] (3.43633703506394,-10.850855075327143)-- (3.0501441756363525,-9.158833603697047);
\draw [line width=.8pt,dash pattern=on 1pt off 1pt] (3.43633703506394,-10.850855075327143)-- (-0.07718175423576668,-9.158833603697047);
\draw [line width=.8pt,dash pattern=on 1pt off 1pt] (-3.9768934029630607,-9.158833603697047)-- (-0.07718175423576668,-9.158833603697047);
\draw [line width=.8pt,dash pattern=on 1pt off 1pt] (-0.46337461366335475,-10.850855075327143)-- (-0.07718175423576668,-9.158833603697047);
\draw [line width=.8pt,dash pattern=on 1pt off 1pt] (-3.9768934029630607,-9.158833603697047)-- (-2.89480505683453,-7.801937735804837);
\draw [line width=.8pt,dash pattern=on 1pt off 1pt] (3.43633703506394,-10.850855075327143)-- (6.563662964936059,-10.850855075327143);
\draw [line width=.8pt,dash pattern=on 1pt off 1pt] (6.563662964936059,-10.850855075327143)-- (5,-11.603875471609676);

\draw [line width=1pt] (-3.9768934029630616,-9.158833603697047)-- (-5.712428359433295,-9.158833603697047);

\draw [line width=1pt] (-5.712428359433295,-9.158833603697047)-- (-8.530051662032058,-7.801937735804838);

\draw [line width=1pt] (-8.530051662032058,-7.801937735804838)-- (-6.794516705561825,-7.801937735804838);

\draw [line width=.8pt,dash pattern=on 1pt off 1pt] (-6.794516705561825,-7.801937735804838)-- (-5.712428359433295,-9.158833603697047);

\draw [line width=1pt] (-6.794516705561824,-7.801937735804836)-- (-2.89480505683453,-7.801937735804837);

\draw [line width=1pt] (-2.89480505683453,-7.801937735804837)-- (-0.07718175423576668,-9.158833603697047);

\draw [line width=1pt] (-0.07718175423576668,-9.158833603697047)-- (3.0501441756363525,-9.158833603697047);

\draw [line width=1pt] (3.0501441756363525,-9.158833603697047)-- (6.563662964936059,-10.850855075327143);

\draw [line width=1pt] (6.563662964936059,-10.850855075327143)-- (8.127325929872118,-11.603875471609676);
\draw [line width=1pt] (5,-11.603875471609676)-- (3.43633703506394,-10.850855075327143);
\draw [line width=1pt] (3.43633703506394,-10.850855075327143)-- (-0.46337461366335475,-10.850855075327143);
\draw [line width=1pt] (-0.46337461366335475,-10.850855075327143)-- (-3.9768934029630607,-9.158833603697047);
\draw [line width=1pt] (-6.794516705561824,-7.801937735804836)-- (-3.9768934029630607,-9.158833603697047);

\end{tikzpicture}
		\begin{tikzpicture}[line cap=round,line join=round,>=triangle 45,x=1cm,y=1cm, scale=1.7]
		\clip(-3.6,-1) rectangle (5.473333333333334,3);
		\draw [line width=1pt]  (-1.4088116512993818,2.1906431337674115)-- (-0.9749279121818236,2.1906431337674115);
		\draw [line width=1pt]  (-1.4088116512993818,2.1906431337674115)-- (-1.4088116512993818,1.4088116512993818);
		\draw [line width=1pt,dash pattern=on 1pt off 1pt] (-0.9749279121818236,2.1906431337674115)-- (-0.9749279121818236,1.4088116512993818);
		\draw [line width=1pt,dash pattern=on 1pt off 1pt] (-0.9749279121818236,1.4088116512993818)-- (0,1.4088116512993818);
		\draw [line width=1pt,dash pattern=on 1pt off 1pt] (0,1.4088116512993818)-- (0,0.4338837391175581);
		\draw [line width=1pt,dash pattern=on 1pt off 1pt] (0,0.4338837391175581)-- (0.7818314824680298,0.4338837391175581);
		\draw [line width=1pt] (0,0)-- (0.7818314824680298,0);
		\draw [line width=1pt] (0.7818314824680298,0)-- (0.7818314824680298,0.4338837391175581);
		\draw [line width=1pt] (0,0)-- (0,0.4338837391175581);
		\draw [line width=1pt] (0.7818314824680298,0.4338837391175581)-- (0.7818314824680298,1.4088116512993818);
		\draw [line width=1pt] (0,0.4338837391175581)-- (-0.9749279121818236,0.4338837391175581);
		\draw [line width=1pt] (-0.9749279121818236,0.4338837391175581)-- (-0.9749279121818236,1.4088116512993818);
		\draw [line width=1pt] (-0.9749279121818236,1.4088116512993818)-- (-1.4088116512993818,1.4088116512993818);
		\draw [line width=1pt] (0,1.4088116512993818)-- (0.7818314824680298,1.4088116512993818);
		\draw [line width=1pt] (0,1.4088116512993818)-- (0,2.1906431337674115);
		\draw [line width=1pt] (0,2.1906431337674115)-- (-0.9749279121818236,2.1906431337674115);
		\draw [line width=1pt] (1.1643590901911454,2.8356409098088546)-- (1.1643590901911454,2.1928533001223154);
		\draw [line width=1pt] (1.1643590901911454,2.1928533001223154)-- (1.506379233516814,2.1928533001223154);
		\draw [line width=1pt] (1.506379233516814,2.1928533001223154)-- (1.506379233516814,1.2080455471101073);
		\draw [line width=1pt] (1.506379233516814,1.2080455471101073)-- (2.3724046373012526,1.2080455471101073);
		\draw [line width=1pt] (2.3724046373012526,1.2080455471101073)-- (2.3724046373012526,0.3420201433256687);
		\draw [line width=1pt] (2.3724046373012526,0.3420201433256687)-- (3.3572123903134607,0.3420201433256687);
		\draw [line width=1pt] (3.3572123903134607,0.3420201433256687)-- (3.3572123903134607,0);
		\draw [line width=1pt] (3.3572123903134607,0)-- (4,0);
		\draw [line width=1pt] (4,0)-- (4,0.3420201433256687);
		\draw [line width=1pt] (4,0.3420201433256687)-- (4,1.2080455471101073);
		\draw [line width=1pt] (4,1.2080455471101073)-- (3.3572123903134607,1.2080455471101073);
		\draw [line width=1pt] (3.3572123903134607,1.2080455471101073)-- (3.3572123903134607,2.1928533001223154);
		\draw [line width=1pt] (3.3572123903134607,2.1928533001223154)-- (2.3724046373012526,2.1928533001223154);
		\draw [line width=1pt] (2.3724046373012526,2.1928533001223154)-- (2.3724046373012526,2.8356409098088546);
		\draw [line width=1pt] (2.3724046373012526,2.8356409098088546)-- (1.506379233516814,2.8356409098088546);
		\draw [line width=1pt] (1.506379233516814,2.8356409098088546)-- (1.1643590901911454,2.8356409098088546);
		\draw [line width=1pt,dash pattern=on 1pt off 1pt] (1.506379233516814,2.1928533001223154)-- (1.506379233516814,2.8356409098088546);
		\draw [line width=1pt,dash pattern=on 1pt off 1pt] (1.506379233516814,2.1928533001223154)-- (2.3724046373012526,2.1928533001223154);
		\draw [line width=1pt,dash pattern=on 1pt off 1pt] (2.3724046373012526,2.1928533001223154)-- (2.3724046373012526,1.2080455471101073);
		\draw [line width=1pt,dash pattern=on 1pt off 1pt] (2.3724046373012526,1.2080455471101073)-- (3.3572123903134607,1.2080455471101073);
		\draw [line width=1pt,dash pattern=on 1pt off 1pt] (3.3572123903134607,1.2080455471101073)-- (3.3572123903134607,0.3420201433256687);
		\draw [line width=1pt,dash pattern=on 1pt off 1pt] (3.3572123903134607,0.3420201433256687)-- (4,0.3420201433256687);
		\draw [line width=1pt] (-2,0)-- (-2.9510565162951536,0);
		\draw [line width=1pt] (-2.9510565162951536,0)-- (-2.9510565162951536,0.5877852522924731);
		\draw [line width=1pt] (-2.9510565162951536,0.5877852522924731)-- (-3.538841768587627,0.5877852522924731);
		\draw [line width=1pt] (-3.538841768587627,0.5877852522924731)-- (-3.538841768587627,1.5388417685876266);
		\draw [line width=1pt] (-3.538841768587627,1.5388417685876266)-- (-2.9510565162951536,1.5388417685876266);
		\draw [line width=1pt] (-2.9510565162951536,1.5388417685876266)-- (-2,1.5388417685876266);
		\draw [line width=1pt] (-2,1.5388417685876266)-- (-2,0);
		\draw [line width=1pt] (-3.538841768587627,1.5388417685876266)-- (-3.538841768587627,0.5877852522924731);
		\draw [line width=1pt,dash pattern=on 1pt off 1pt] (-2.951056516295153,0.5877852522924728)-- (-2,0.5866666666666664);
		\draw [line width=1pt,dash pattern=on 1pt off 1pt] (-2.951056516295153,0.5877852522924728)-- (-2.9510565162951536,1.5388417685876266);
		\begin{scriptsize}
		\draw[color=black] (0.4333333333333327,-0.1) node {$\alpha_1$};
		
		\draw[color=black] (0.9,0.2) node {$\beta_3$};

		\draw[color=black] (0.9,0.9) node {$\beta_2$};
		
		\draw[color=black] (-0.44666666666666754,0.3) node {$\alpha_2$};
		\draw[color=black] (-1.14,1.3) node {$\alpha_3$};
		\draw[color=black] (0.1,1.8) node {$\beta_1$};

		\draw[color=black] (1.3,2.1) node {$\alpha_4$};
		\draw[color=black] (2.02,1.1) node {$\alpha_3$};
		\draw[color=black] (2.953333333333333,0.2) node {$\alpha_2$};
		\draw[color=black] (3.7666666666666666,-0.1) node {$\alpha_1$};
		
		\draw[color=black] (4.2,0.2) node {$\beta_4$};
		\draw[color=black] (4.2,0.8) node {$\beta_3$};
		
		\draw[color=black] (3.5,1.7) node {$\beta_2$};
		
		\draw[color=black] (2.5,2.5) node {$\beta_1$};

		\draw[color=black] (2.02,3.2466666666666657) node {$a_1$};
		\draw[color=black] (1.42,3.2466666666666657) node {$b_1$};
		
		\draw[color=black] (-1.9,0.4) node {$\beta_2$};
		
		\draw[color=black] (-2.5,-0.1) node {$\alpha_1$};
		
		\draw[color=black] (-3.3,0.5) node {$\alpha_2$};
		\draw[color=black] (-1.9,0.9666666666666663) node {$\beta_1$};
		
		\draw[color=black] (-3.66,1.26) node {$q_2$};
		
		\end{scriptsize}
		\end{tikzpicture}

\caption{Above, the double $(2m+1)-$gon $X_0$ is cut into $2(2m-1)$ triangles, which are re-arranged into a slanted stair-shape, whose slanted  sides are then rotated and sheared to create the right-angled stair-shape $S_0$. Below, a staircase model $S_0$ for  $X_0$ ($n=2m+1$), for $m=2,3,4$, with parameters 
$(\alpha_i,\beta_i)_{i=1,\dots,m}$. The surface $S_0$ for $m=2$ is usually shown rotated by $180$ degrees, as the golden L (see~\cite{DFT},~\cite{DL18}).}
 \label{heptagon}
\end{center}
\end{figure}
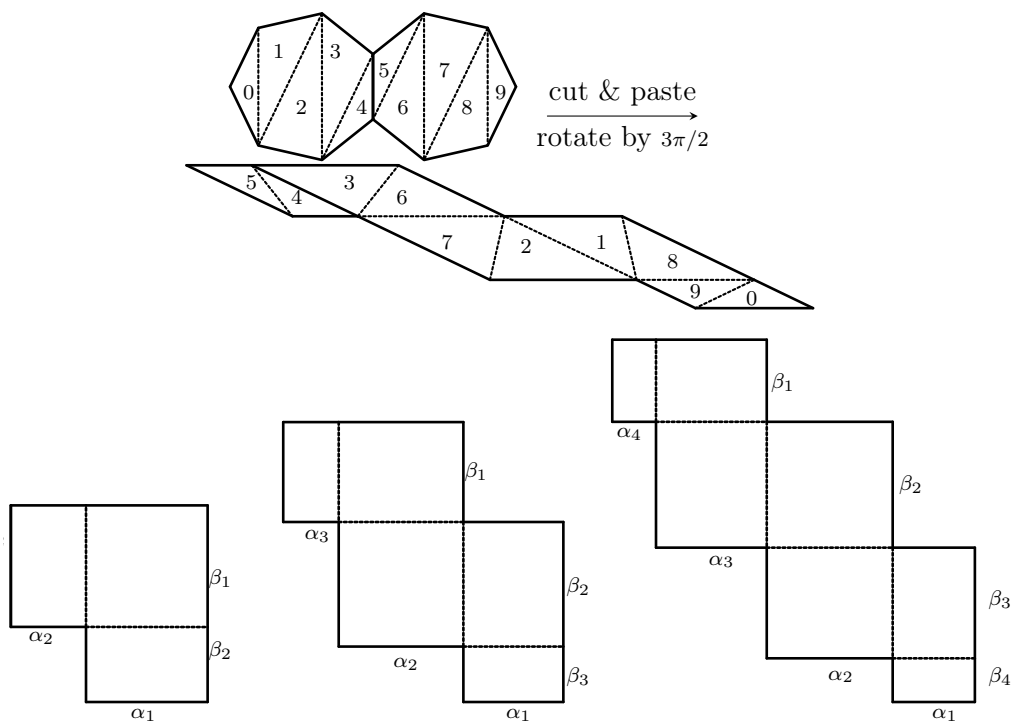
In order to exhibit the Veech group and a fundamental domain, one needs to compute the dimensions of the $m$ horizontal cylinders of $S_0$. This immediately follows 
from~\cite[\S5]{Veech}: for $k=1,\dots,m$, the core curve of the kth cylinder of $S_0$ is given by $\alpha_{k-1}+\alpha_{k}$ (with the dummy condition $\alpha_{0}:=0$), and its height is $\beta_{m-k+1}$. From~\cite[\S5]{Veech} we draw:
\begin{equation}
\label{eq:lengths}
l(\alpha_k)=l(\beta_k) = \sin \frac{2k\pi}{n}, \textrm{ for any } k=1,\dots,m, \ \mathrm{and}\ \Vol(S_0) = \frac{n}{2}\cos \frac{\pi}{n}.
\end{equation}
In particular all moduli are the same and equal to
$$
\frac{{\mathrm height}}{{\mathrm width}} =\frac{\sin \frac{2(m-k+1)\pi}{n}}{ \sin \frac{2(k-1)\pi}{n} + \sin \frac{2k\pi}{n}} = 
\frac{\sin \frac{(2k-1)\pi}{n}}{ 2\sin \frac{2(k-1+k)\pi}{2n}\cos\frac{2(k-1-k)\pi}{2n}} = \frac1{2\cos \pi/n}.
$$
From this knowledge, it is easy to construct affine homeomorphisms of $S_0$ whose derivative maps (outside the singularity) are parabolic elements. We describe the Veech group of $S_0$ in the next section.

\subsection{Veech group and fundamental domain}
For $n\geq 3$, we denote by $\Gamma_n$ the Hecke triangle group of level $n$ (or signature $(2,n,\infty)$) generated by 
\[
T=\left(
\begin{array}{cc}
1 & \Phi_n \\
0 & 1 
\end{array}
\right)
\mbox{ and  }
R=
\left(
\begin{array}{cc}
0 & -1 \\
1 & 0 
\end{array}
\right),
\mbox{ setting } \Phi_n = 2 \cos \frac{\pi}{n}. 
\]
In the following we will simply use the notation $\Phi$ for $ \Phi_n $. The group $\Gamma_n$, acting on the hyperbolic plane $\Hyp^2$, has a fundamental domain $\mathcal D$ depicted in Figure~\ref{fundamendal:domain}. It is comprised between the vertical geodesics with abscissae $-\Phi /2$ and $\Phi /2$, and the geodesic with endpoints $\pm 1$. The Veech group of $S_0$ coincides with $\Gamma_n$. In the fundamental domain the staircase model $S_0$ is represented by the point $i$, while $X_0$ corresponds to the lower corners of $\mathcal D$ (the intersection between a vertical boundary and the circular boundary).
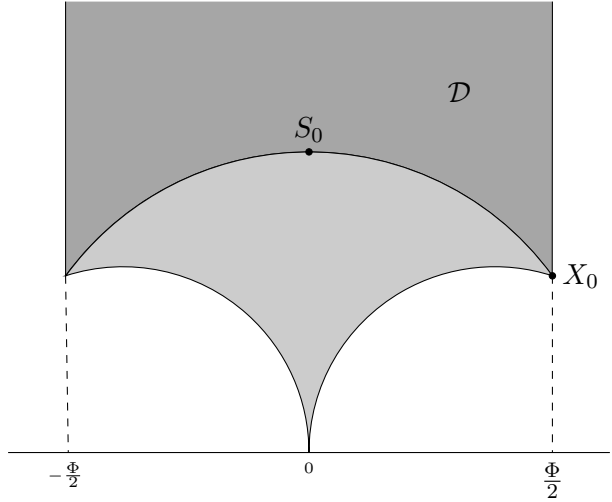
\begin{figure} [h!]
\begin{center}
\begin{tikzpicture}[scale=4,every to/.style={hyperbolic plane}]

\def\h{1.5}
\coordinate (b) at (-3,0);
\coordinate (dpm) at (-0.8090,0.5878);
\coordinate (dpmbis) at (0.8090,0.5878);

\draw (-1,0) -- (1,0);

\filldraw [draw=black, fill=gray!70](-0.8090,\h) to (dpm) to (dpmbis) -- (0.8090,\h);
\filldraw [draw=black, fill=gray!40] (dpm) to (0,0) to (dpmbis) to cycle;

\draw[dashed] (dpmbis) -- (0.8090,0);

\draw[dashed] (dpm) -- (-0.80,0);

\node at (0.5,1.2)  {$\mathcal D$};
\filldraw[black] (dpmbis) circle (0.3pt) node[right]{$X_0$};
\filldraw[black] (0,1) circle (0.3pt) node[above]{$S_0$};

\node[below] at (0,0)  {$\scriptscriptstyle  0$};
\node[below] at (-0.8090,0)  {$\scriptscriptstyle -\frac{\Phi}{2}$};
\node[below] at (0.8090,0)  {$\frac{\Phi}{2}$};
\end{tikzpicture}
\end{center}
\caption{A fundamental domain $\mathcal D$ of the Veech group of the staircase model of the double regular  $(2m+1)$-gon, along with its reflection in the geodesic $(-1,1)$.
}\label{fundamendal:domain}
\end{figure}

\section{Boundedness of $\KVol$ on Teichm\"uller discs}\label{section:boundedness}
Before studying the maximum of $\KVol$ on Teichm\"uller discs, we give a criterion to ensure that it is indeed a bounded function. We conclude this section with a proof of Theorem~\ref{thm:bound}.

We can first notice that if there are two parallel saddle connections $\alpha,\beta$ on $X\in \mathcal H(n-3)=\mathcal H(2m-2)$ having non trivial intersection then applying the Teichm\"uller geodesic flow in the orthogonal direction of $\alpha,\beta$, we get for all $t>0$
$$
\KVol(g_tX) \geq \frac{1}{e^{-2t}|\alpha||\beta|}.
$$
Thus $\KVol$ is not bounded on the Teichm\"uller disc of $X$. Actually, this remark applies to a large class of translation surfaces (of genus at least two): those decomposed into a single metric cylinder, as we will see below. \newline 

Before proving Lemma~\ref{lm:strebel} we recall the notion of the angle of a horizontal saddle connection on a surface whose horizontal foliation has only closed leaves. The union of all (horizontal) saddle connections and the singularity defines a finite oriented graph $\Gamma$. Orientation on the edges comes from the canonical orientation of the horizontal foliation. At the singularity $p$ the direction of saddle connections attached to $p$ alternates (between directions toward $p$ and from $p$) as we follow the clockwise order. For a saddle connection $\gamma$ one can count the number of different sectors between the two directions it determines. This gives an angle $(2k+1)\pi$ well defined modulo $2(2m-1)\pi$. Observe that if $\gamma$ has angle $\pi$ then it is the boundary of a metric cylinder embedded into the surface.

\begin{Lem}
\label{lm:strebel}
If $X$ has a cylinder decomposition  all of whose boundary saddle connections  have trivial intersection pairwise, then there is a saddle connection with  angle $\pi$. In particular if $X$ has a one cylinder decomposition and $\mathrm{genus}(X)>1$ then there are two parallel saddle connections intersecting non trivially.
\end{Lem}
\begin{proof}
We consider the saddle connection $\gamma$ having the smallest angle $(2k+1)\pi$ at $p$ (see Figure~\ref{fig:separatrix}). Assume $k>0$. Then $\gamma$ determines $2k+1$ sectors and so $2k$ saddle connections $\beta_1,\dots,\beta_{2k}$. Since the angle of $\gamma$ is minimal, no $\beta_i$ can begin and end inside the sector of angle $(2k+1)\pi$ cut by $\gamma$. Therefore the intersection of $\beta_i$ with $\gamma$ is non trivial for every $i=1,\dots,k$.
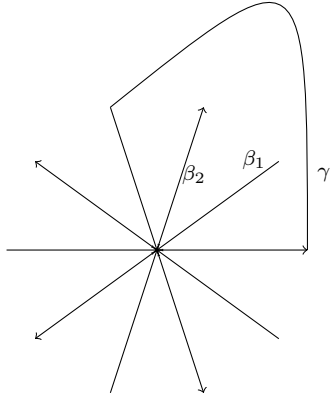
\begin{figure}[h]

		\begin{center}
		\begin{tikzpicture}[scale=2]
		
 \draw [-to] (0,0) --(1.0, 0.00);
 \draw [to-] (0,0) --(0.81, 0.59);
 \draw [-to] (0,0) --(0.31, 0.95);
 \draw [to-] (0,0) --(-0.31, 0.95);
 \draw [-to] (0,0) --(-0.81, 0.59);
 \draw [to-] (0,0) --(-1.0, 0.00);
 \draw [-to] (0,0) --(-0.81, -0.59);
 \draw [to-] (0,0) --(-0.31, -0.95);
 \draw [-to] (0,0) --(0.31, -0.95);
 \draw [to-] (0,0) --(0.81, -0.59);

 \draw (1.0, 0.00) .. controls (1,2) .. (-0.31, 0.95);
 \draw (1,0.5) node[right]{$\scriptstyle{\gamma}$};		
 \draw (0.5,0.6) node[right]{$\scriptstyle{\beta_1}$};		
  \draw (0.1,0.5) node[right]{$\scriptstyle{\beta_2}$};		
		\end{tikzpicture}
	\end{center}
	\caption{A separatrix diagram in $\mathcal H(6)$. The singularity $p$ has conical angle $10\pi$ and $\alpha$ has angle $3\pi$ (or $7\pi$ modulo $10\pi$)}
	\label{fig:separatrix}
\end{figure}
\end{proof}
From this observation and the fact that one-cylinder surfaces are dense, we immediately deduce that $\KVol$ is not bounded on any connected component of $\mathcal H(2m-2)$ for $m\geq 2$. Similarly, since $\KVol$ is a continuous function, it is not bounded on the Teichm\"uller disc of a generic (with respect to the Masur--Veech measure) surface $X$. Veech surfaces are exceptionally symmetric translation surfaces and are not generic if $m\geq 2$. For those surfaces one has the following converse.
\begin{Prop}
\label{prop:veech:bounded}
$\KVol$ is bounded on the Teichm\"uller disc of a Veech surface in $\mathcal H(2m-2)$ if there are no parallel saddle connections intersecting non trivially.
\end{Prop}
\begin{proof}[Proof of Proposition~\ref{prop:veech:bounded}]
Let us consider only surfaces which have total area $1$. Let $X$ be a Veech surface. Let $\theta$ be the angle associated to two periodic directions $(d,d')$ having saddle connections $\alpha,\beta$ with nontrivial intersections. The hypothesis ensures that $d \neq d'$, so that up to swapping $d$ and $d'$ we may assume  $\theta \in ]0,\pi[$. Then, any intersection between a saddle connection $\alpha$ with direction $d$, and a saddle connection $\beta$ with direction $d'$, if it occurs outside the singularity, is positive. Therefore, given two saddle connections $\alpha$ and $\beta$, with respective directions $d$ and $d'$, either $\Int (\alpha, \beta)=-1$, in which case $\alpha$ and $\beta$ intersect only once, at the singularity, or $\Int (\alpha, \beta) \geq 0$.

By Veech dichotomy, $X$ is decomposed into cylinders $C_1,\dots,C_r$ (of heights $h_1(d),\dots,h_r(d)$) with direction $d$ and saddle connections $\alpha_i$. Observe that in the minimal stratum $\mathcal H(2m-2)$ the number of cylinders $r$ is bounded by the genus $m$ of the surface. We can subdivide $\beta$ by looking at the intersections $\Int ( \alpha_{k}, \beta)$. This gives:
\begin{equation}\label{cusp_identity_double}
l(\beta)\cdot \sin \theta = \sum_{k=1}^r  h_k(d) \Int ( \alpha_{k}, \beta) = \sum_{k=1}^r  c_k(d) \cfrac{\Int ( \alpha_{k}, \beta)}{l(\alpha_{k})},
\end{equation}
where $c_k(d)=h_k(d) \cdot l(\alpha_{k})$ represents the area of an embedded parallelogram supported on $\alpha_{k}$. We write $l(\alpha_{k})l(\beta) = \alpha_{k} \wedge \beta/\sin \theta$, so
$$
1 = \sum_{k=1}^r  c_k(d) \cfrac{\Int ( \alpha_{k}, \beta)}{\alpha_{k} \wedge \beta}.
$$
Writing in a slightly different way:
$$
1 = \sum_{\Int > 0}  c_k(d) \cfrac{\Int ( \alpha_{k}, \beta)}{\alpha_{k} \wedge \beta} - \sum_{\Int =-1}   \cfrac{c_k(d)}{\alpha_{k} \wedge \beta}.
$$
By~\cite{Vorobets}, there are no small triangles in $X$ {\em i.e.}  there exists $M>0$, depending only on $X$, such that $|\alpha_{k} \wedge \beta| > M$ (this is actually the easy part of the characterization of Veech surfaces in \cite{SW}). In particular since $c_k(d) \leq \mathrm{Area(X)}=1$, we have 
\[
\sum_{\Int =-1}   \cfrac{c_k(d)}{\alpha_{k} \wedge \beta} < r \cdot \frac{1}{M} < C
\]
 for some uniform constant $C=C(X)$ (recalling $r\leq m)$. Thus
\[
\cfrac{\Int ( \alpha_{k}, \beta)}{\alpha_{k} \wedge \beta}  < \cfrac{1+C}{c_k(d)} < \cfrac{1+C}{M}.
\]
Hence $\cfrac{\Int ( \alpha_{k}, \beta)}{l(\alpha_{k})l(\beta)}$ is uniformly bounded on the Teichm\"uller disc of $X$. Since $\mathrm{Area}(X)=1$, $\KVol$ is uniformly bounded as well.
\end{proof}
In $\mathcal{H}(2)$ we have a more precise description.
\begin{proof}[Proof of Theorem~\ref{thm:bound}]
By~\cite{Mc} Teichm\"uller discs are either dense or closed. So if $X$ is not a Veech surface, $\KVol$ is not bounded from above on its Teichm\"uller disc. For Veech surfaces, a quick inspection of possible cylinder diagrams leads to the following observation: for two-cylinder decompositions, there are no parallel saddle connections with non trivial intersections. Now a Veech surface in genus two is either primitive or square-tiled. By~\cite[Corollary A.2]{Mc:spin}, square-tiled surfaces all admit one-cylinder decompositions. Since primitive Veech surfaces in $\mathcal H(2)$ do not have one-cylinder decompositions, we get the desired result.
\end{proof}

\section{$\mathrm{SL}_2(\mathbb{R})-$action and directions in the Teichm\"uller disc}\label{sec:SL2R}

The action of $\mathrm{SL}_2(\mathbb R)$ on moduli spaces provides a powerful tool to study the dynamics of the translation flow on $X$. This group is also acting affinely on surfaces, preserving the intersection form. Saddle connections are mapped to saddle connections, but  lengths are not preserved in general. The right quantity to consider is not the length of $\alpha$ or $\beta$ but rather the quantity $\overrightarrow{\alpha} \wedge \overrightarrow{\beta}=l(\alpha)l(\beta)\sin \theta$, where $\theta$ is the angle between the holonomy vectors $\overrightarrow{\alpha},\overrightarrow{\beta}$ associated to $\alpha,\beta$. The wedge product is invariant under $\mathrm{SL}_2(\mathbb R)$ and is twice the area of a (virtual) triangle delimited by $\alpha$ and $\beta$. \medskip

This observation motivates the following definition.
\begin{Def}
For $d \in \RR P^1$, we say that a saddle connection in $S_0$  has direction $d$ if it has direction $d$ in the plane template of Figure \ref{heptagon}. For $M \in \mathrm{GL}_2^+ (\RR)$ we say that a saddle connection $\alpha$ in $M\cdot S_0$ has direction $d$ if $M^{-1}\cdot \alpha$ has direction $d$ in $S_0$.
\end{Def}
This is a bit counter-intuitive because $\alpha$ may not have direction $d$ in a plane template for $M\cdot S_0$, but it makes sense with the following \medskip

\begin{Prop}\label{prop:banane}
Using the identifications \footnote{Note that $\chi$ defines a right action of $ GL_2^+(\RR)$. See~\cite[Section 6.1]{survey_Massart},  as to why we should quotient by $SO_2(\RR)$ on the left, and act by $ GL_2^+(\RR)$ on the right.}
\[ \Psi : d = [x:y] \in \RR P^1 \mapsto -\frac{x}{y} \in \RR \cup \{ \infty \} \equiv \partial \HH \]
and, for $M = \begin{pmatrix} a & b \\ c & d \end{pmatrix} \in SO_2(\RR) \backslash GL_2^+(\RR)$,
\[\chi : M \cdot S_0 \in \mathcal T_n 
\mapsto \frac{di+b}{ci+a} \in \HH, \]
the locus of surfaces in $\mathcal T_n$ where the directions $d$ and $d'$ make an (unoriented) angle $\theta \in ]0, \frac{\pi}{2}]$ is the banana neighbourhood
\[
\gamma_{d,d',r}= \{ z \in \Hyp^2: d_{\mathrm{hyp}}(z, \gamma_{d,d'} )=r \}
\]
where $\cosh r = \cfrac{1}{\sin \theta}$, $\gamma_{d,d'}$ denotes the hyperbolic geodesic with endpoints $\Psi (d)$ and $\Psi (d')$, and $d_{hyp}$ is the hyperbolic distance.\newline

In particular, the locus of surfaces in $\mathcal T_n$ where the directions $d$ and $d'$ are orthogonal is the hyperbolic geodesic with endpoints $\Psi (d)$ and $\Psi (d')$. \newline
\end{Prop}

\begin{Nota}\label{nota:angle}
In the rest of the paper, we denote by $\theta(X,d,d')$ the angle between the directions $d$ and $d'$ in the surface $X = M \cdot S_0$. Notice that Proposition \ref{prop:banane} gives
$$
\cosh d_{hyp}(X, \gamma_{d,d'}) = \frac{1}{\sin \theta(X,d,d')}.
$$
\end{Nota}

\begin{proof}[Proof of Proposition~\ref{prop:banane}]
For $\theta \in ]0, \frac{\pi}{2}]$ and $u,v\in \RR^2$ with equivalence classes $d\neq d'$ we define
\[
\mathcal{M}(d,d',\theta) = \{ M \in GL_2^+ (\RR): \mbox{min}( \mbox{angle}(\pm Mu, \pm Mv),  \mbox{angle}(\pm Mv,\pm Mu)) = \theta \}.
\]
Observe that $\mathcal{M}(d,d',\theta)$ is well defined (it only depends on the equivalence classes $d,d'$ because the angles are taken modulo $\pi$) and equivariant by right multiplication: if $G \in GL_2^+ (\RR)$ then
\begin{equation}
\label{eq:invariance}
\mathcal{M}(d,d',\theta).G=\mathcal{M}(G^{-1}d,G^{-1}d',\theta).
\end{equation}
Denote $\bar{\mathcal{M}}(d,d',\theta)$ the projection of $\mathcal{M}(d,d',\theta)$ to $\Hyp^2$. Observe that
$\mathcal{M}(d,d',\theta)$ is invariant by left multiplication by $SO_2(\RR)$, so any matrix in $GL_2^+(\RR)$ that projects to an element of $\bar{\mathcal{M}}(d,d',\theta)$, is actually in $\mathcal{M}(d,d',\theta)$.\newline

Let us look at the case $d=\bar{\left(\begin{smallmatrix}1 \\ 0 \end{smallmatrix}\right)}=\infty$ and $d'=\bar{\left(\begin{smallmatrix}0 \\ 1 \end{smallmatrix}\right)}=0$. Observe that in that case $\bar{\mathcal{M}}(d,d',\theta)$ is invariant by $z \mapsto \lambda z$, for any $\lambda >0$. Indeed, take $\lambda >0$ and 
$z \in \bar{\mathcal{M}}(d,d',\theta)$, and let
\[
M=
\left(
\begin{array}{cc}
a & b \\
c & d
\end{array}
\right)
\]
be an element of $\mathcal{M}(d,d',\theta) \subset GL_2^+(\RR)$ which projects to $z$. Then the matrix 
\[
M'=
\left(
\begin{array}{cc}
a & b \\
c & d
\end{array}
\right)
\left(
\begin{array}{cc}
1 & 0 \\
0 & \lambda
\end{array}
\right)
\in GL_2^+(\RR)
\]
projects to $\lambda z$. But the equivalence class, in $\RR P^1$, of $\left(\begin{smallmatrix} 1 & 0 \\ 0 & \lambda \end{smallmatrix}\right) \left(\begin{smallmatrix}1 \\ 0 \end{smallmatrix}\right)$ (resp. $\left(\begin{smallmatrix}1 & 0 \\ 0 & \lambda \end{smallmatrix}\right) \left(\begin{smallmatrix}0 \\ 1 \end{smallmatrix}\right)$), is $d$ (resp. $d'$), and we have seen that $\mathcal{M}(d,d',\theta)$ only depends on the equivalence classes $d,d'$, so $M \in \mathcal{M}(d,d',\theta)$ entails $M' \in \mathcal{M}(d,d',\theta)$. Therefore $\lambda z \in \bar{\mathcal{M}}(d,d',\theta)$.

\begin{figure}[h]
\hskip -55mm
\begin{minipage}[l]{0.1\linewidth}
\begin{tikzpicture}[line cap=round,line join=round,>=triangle 45,x=1cm,y=1cm, scale=1.25]
\clip(-3,0) rectangle (2.6,4);
\draw [line width=2pt,color=gray,domain=-3.0190928292046935:2.895707170795306] plot(\x,{(-0-0*\x)/4});
\draw [shift={(0,0)},line width=1pt,color=black]  plot[domain=0:3.141592653589793,variable=\t]({1*2*cos(\t r)+0*2*sin(\t r)},{0*2*cos(\t r)+1*2*sin(\t r)});

\draw [->,line width=1pt,color=red] (0,0) -- (1,1);
\draw [line width=1pt,color=red] (0,0) -- (5,5);

\draw [->,line width=1pt,color=black] (0,0) -- (0,1);
\draw [line width=1pt,color=black] (0,0) -- (0,5);

\draw [->,line width=1pt,color=red] (0,0) -- (-1,1);
\draw [line width=1pt,color=red] (0,0) -- (-5,5);

\draw (-2.001161147327249,0) node[anchor=north west] {$\bar{u}$};
\draw (1.9857379400260755,0) node[anchor=north west] {$\bar{v}$};
\draw (0,3) node[anchor=north west] {$\gamma_{d,d'}=\gamma_{0,\infty}$};
\draw (-2.4,3) node[anchor=north west] {$\gamma_{d,d',r}$};
\draw (1,2.5) node[anchor=north west] {$\gamma_{d,d',r}$};

\draw (1,2.5) node[anchor=north west] {$\gamma_{d,d',r}$};

\node[draw,circle,inner sep=1.5pt,fill] at (0,2) {};
\node[draw,circle,inner sep=1.5pt,fill] at (1.4,1.4) {};
\draw[left] (0,2.2) node {\begin{scriptsize} $i$ \end{scriptsize}};
\draw[right] (1.4,1.4) node {\begin{scriptsize} $z=e^{i\theta}$ \end{scriptsize}};

\draw [shift={(0,0)},line width=1pt,fill=red,fill opacity=0.10000000149011612] (0,0) -- (45.13010235415601:0.23134810951760104) arc (45.13010235415601:0:0.23134810951760104) -- cycle;

\draw [shift={(0,0)},line width=1pt,fill=red,fill opacity=0.10000000149011612] (0,0) -- (-45.13010235415601:-0.23134810951760104) arc (-45.13010235415601:0:-0.23134810951760104) -- cycle;

\begin{scriptsize}
\draw[color=black] (-0.3,0.15) node {$\theta$};
\draw[color=black] (0.3,0.15) node {$\theta$};
\end{scriptsize}
\end{tikzpicture}
\end{minipage}
\hskip 60mm
\begin{minipage}[l]{0.1\linewidth}
\begin{tikzpicture}[line cap=round,line join=round,>=triangle 45,x=1cm,y=1cm, scale=1.25]
\clip(-3,0) rectangle (2.6,4);
\draw [shift={(0,0)},line width=1pt,color=black]  plot[domain=0:3.141592653589793,variable=\t]({1*2*cos(\t r)+0*2*sin(\t r)},{0*2*cos(\t r)+1*2*sin(\t r)});
\draw [shift={(0,1.5)},line width=1pt,color=red]  plot[domain=-0.6435011087932843:3.7850937623830774,variable=\t]({1*2.5*cos(\t r)+0*2.5*sin(\t r)},{0*2.5*cos(\t r)+1*2.5*sin(\t r)});
\draw [shift={(0,-1.5)},line width=1pt,color=red]  plot[domain=0.6435011087932844:2.498091544796509,variable=\t]({1*2.5*cos(\t r)+0*2.5*sin(\t r)},{0*2.5*cos(\t r)+1*2.5*sin(\t r)});
\draw [->,line width=1pt,color=red] (-2,0) -- (-1.4084909621903532,0.7886787170795292);
\draw [->,line width=1pt,color=black] (-2,0) -- (-2,1);
\draw [->,line width=1pt,color=red] (-2,0) -- (-2.6164052672750966,0.8218736897001288);
\draw [line width=2pt,color=gray,domain=-3.0190928292046935:2.895707170795306] plot(\x,{(-0-0*\x)/4});
\draw (-2.001161147327249,0) node[anchor=north west] {$\bar{u}$};
\draw (1.9857379400260755,0) node[anchor=north west] {$\bar{v}$};
\draw (-0.4973984354628422,2.5) node[anchor=north west] {$\gamma_{d,d'}$};
\draw (-0.3971475880052151,1.0268959582790067) node[anchor=north west] {$\gamma_{d,d',r}$};
\draw (-0.28918513689700126,3.9727285528031264) node[anchor=north west] {$\gamma_{d,d',r}$};

\draw [shift={(-2,0)},line width=1pt,fill=red,fill opacity=0.10000000149011612] (0,0) -- (53.13010235415601:0.23134810951760104) arc (53.13010235415601:0:0.23134810951760104) -- cycle;

\draw [shift={(-2,0)},line width=1pt,fill=red,fill opacity=0.10000000149011612] (0,0) -- (-53.13010235415601:-0.23134810951760104) arc (-53.13010235415601:0:-0.23134810951760104) -- cycle;

\begin{scriptsize}
\draw[color=black] (-2.3,0.15) node {$\theta$};
\draw[color=black] (-1.7,0.15) node {$\theta$};
\end{scriptsize}
\end{tikzpicture}
\end{minipage}
\caption{The set $\gamma_{d,d',r}$ for $\cosh r = \frac{1}{\sin \theta}$.}
\label{fig:translation}
\end{figure}
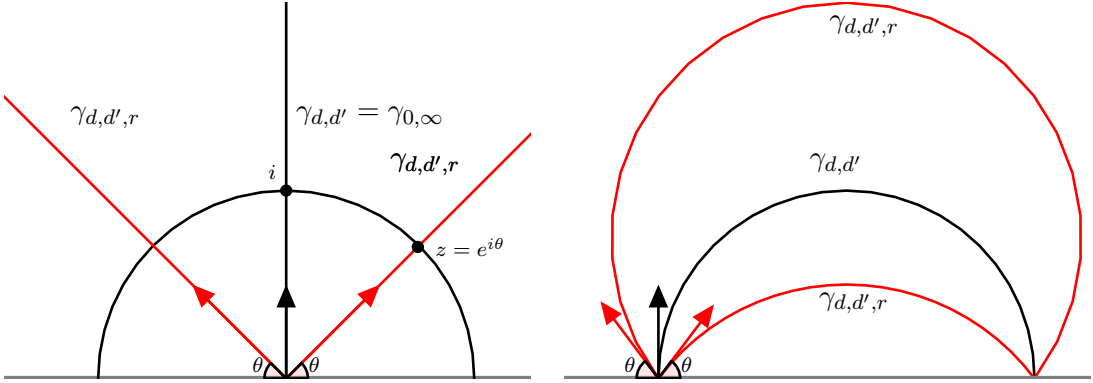
Thus, to determine $\bar{\mathcal{M}}(d,d',\theta) $, it suffices to determine its intersection with the horizontal straight line $\{y=1\}$, which we parametrize as
\[
\Big\{ i + \cot \alpha: \alpha \in \left]0,\pi\right[ \Big\}
\]
A corresponding set of matrices in  $GL_2^+(\RR)$ is given by
\[
\Big\{
\left(
\begin{array}{cc}
1  & \cot \alpha \\
0 & 1
\end{array}
\right)
: \alpha \in  \left] 0,\pi\right[ 
\Big\}
\]
which sends $\left(\begin{smallmatrix}1 \\ 0 \end{smallmatrix}\right)$ and $\left(\begin{smallmatrix}0 \\ 1 \end{smallmatrix}\right)$ to, respectively, $\left(\begin{smallmatrix}1 \\ 0 \end{smallmatrix}\right)$ and $\left(\begin{smallmatrix}\cot \alpha \\ 1 \end{smallmatrix}\right)$. The angle of the latter vectors is $\alpha$, so $\bar{\mathcal{M}}(d,d',\theta) \cap \{y=1\} = \{ \left(\begin{smallmatrix}\cot \theta \\ 1 \end{smallmatrix}\right) \}$. Therefore,  
$\bar{\mathcal{M}}(d,d',\theta)$ is the half-line which starts at the origin, with co-slope $\cot \theta$. This is precisely $\gamma_{d,d',r}$ with $\cosh r = \frac{1}{\sin \theta}$, since the hyperbolic distance from $z=e^{i\theta}$ to the geodesic $\gamma_{0,\infty}$ is realized by the geodesic $\eta$  parameterized by $\eta(t)=e^{it}$ for $t\in [\theta,\pi/2]$, so that by definition
$$
r=\mathrm{d}_{hyp}(z,\gamma_{d,d'}) = l_{\mathrm hyp}(\eta)=\int_\theta^{\pi/2} \frac{dt}{\sin t} = \frac1{2} \log \cfrac{1+\cos \theta}{1-\cos \theta},
$$
thus $\cos \theta = \frac{e^{2r}-1}{e^{2r}+1}$, and
$$
\sin \theta = \sqrt{1-\cos^2 \theta} = \sqrt{1-\left(\frac{e^{2r}-1}{e^{2r}+1}\right)^2}= \cfrac{2e^r}{e^{2r}+1} = \frac{1}{\cosh r}.
$$

Now let us consider the general case. Pick $G\in GL_2^+(\RR)$ taking directions $d,d'$ to 
$\bar{\left(\begin{smallmatrix}1 \\ 0 \end{smallmatrix}\right)}=\infty$ and $\bar{\left(\begin{smallmatrix}0 \\ 1 \end{smallmatrix}\right)}=0$ respectively. Then, by Equation~\ref{eq:invariance}, 
$\mathcal{M}(\infty,0,\theta).G=\mathcal{M}(d,d',\theta)$, so $\bar{\mathcal{M}}(d,d',\theta)$ is the image of $\bar{\mathcal{M}}(\infty,0,\theta)$ by the orientation-preserving isometry of $\Hyp^2$ corresponding to the action of $G$. One verifies that this isometry sends $\infty$ and $0$ to the images of the directions $d$ and $d'$ by the identification $\RR P^1 \simeq \partial \HH$ via the opposite of the co-slope, respectively. This finishes the proof.
\end{proof}

\section{Another look at $\KVol$}\label{sec:another look}

Recall that the function $\KVol$ can be expressed as a supremum over saddle connections $\alpha,\beta$ in~\eqref{eq:KVol}. We will use the invariance of $\mathrm{Int}(\cdot,\cdot)$ for the affine action of $\mathrm{SL}_2(\RR)$ on translation surfaces and the invariance of $\wedge$ for the linear action of $\mathrm{SL}_2(\RR)$ on $\RR^2$ in order to have a more suitable formula to work with. In the sequel we will use the notation $K(X)$:
$$
\KVol(X) = \Vol(X)\cdot K(X).
$$
\begin{Prop}
\label{prop:other:def}
Let $\mathcal{P}$ be the set of periodic directions in $X=M\cdot S_0$ for some $M\in \mathrm{SL}_2 (\RR)$. Then
$$
K(X) = \sup_{
 \begin{scriptsize}
 \begin{array}{c}
d, d' \in \mathcal{P} \\
d \neq d'
\end{array}
\end{scriptsize}} K(d,d')\cdot \sin \theta(X,d,d'),
$$
where
$K(d,d') = \sup_{
 \begin{scriptsize}
 \begin{array}{c}
\alpha\subset S_0\ {\mathrm in\ direction\ } d \\
\beta\subset S_0\ {\mathrm in\ direction\ } d'
\end{array}
\end{scriptsize}
}
\frac{\mathrm{Int} (\alpha,\beta)}{\alpha\wedge \beta}
$ and $\theta(X,d,d')$ is the angle given in Notation~\ref{nota:angle}.
\end{Prop}
Observe that the quantity $K(d,d')$ is invariant under the diagonal action of the Veech group~$\Gamma$. Moreover $\sin \theta(X,d,d') = 1/\cosh r$, where $r$ is the hyperbolic distance between $X$ and the geodesic $\gamma_{d,d'}$, by Proposition \ref{prop:banane}.
\begin{proof}[Proof of Proposition~\ref{prop:other:def}]
Given two saddle connections $\alpha,\beta \subset X$ having directions $d,d'$ (in $X$) and making an angle $\theta$, one has $\alpha \wedge \beta=l(\alpha)l(\beta)\sin \theta$. Notice that parallel saddle connections do not intersect in $X$, so we will assume $d \neq d'$. By definition these saddle connections are the images by $M$ of saddle connections $\alpha',\beta' \subset S_0$ having directions $d, d'$ (in $S_0$), and thus $M\in \mathcal{M}(d,d',\theta)$. In particular the projection of $M$ to $\mathbb H^2$, that is $X\in \bar{\mathcal{M}}(d,d',\theta)$, gives $\theta = \theta(X,d,d')$. Now, by definition (see~\eqref{eq:KVol}),
\begin{multline*}
\sup_{\alpha,\beta}
 \frac{\Int (\alpha,\beta)}{l (\alpha) l (\beta)} = 
 \sup_{ \begin{scriptsize}
 \begin{array}{c}
d, d' \in \mathcal{P} \\
d \neq d'
\end{array}
\end{scriptsize}} 
\sup_{
 \begin{scriptsize}
 \begin{array}{c}
\alpha\subset X\ {\mathrm in\ direction\ } d \\
\beta\subset X\ {\mathrm in\ direction\ } d'
\end{array}
\end{scriptsize}
}
\frac{\mathrm{Int} (\alpha,\beta)}{l (\alpha) l (\beta)} \\
= \sup_{
  \begin{scriptsize}
 \begin{array}{c}
d, d' \in \mathcal{P} \\
d \neq d'
\end{array}
\end{scriptsize}
} 
\sup_{
 \begin{scriptsize}
 \begin{array}{c}
\alpha\subset X\ {\mathrm in\ direction\ } d \\
\beta\subset X\ {\mathrm in\ direction\ } d'
\end{array}
\end{scriptsize}
}
\frac{\mathrm{Int} (\alpha,\beta)}{\alpha \wedge \beta}\cdot \sin \mbox{angle}(\alpha,\beta) \\
= \sup_{ \begin{scriptsize}
 \begin{array}{c}
d, d' \in \mathcal{P} \\
d \neq d'
\end{array}
\end{scriptsize}} \sup_{
 \begin{scriptsize}
 \begin{array}{c}
M^{-1}\alpha\subset S_0\ {\mathrm in\ direction\ } d \\
M^{-1}\beta\subset S_0\ {\mathrm in\ direction\ } d'
\end{array}
\end{scriptsize}
}
\frac{\mathrm{Int} (M^{-1}\alpha,M^{-1}\beta)}{M^{-1}\alpha\wedge M^{-1}\beta} \cdot \sin \theta(X,d,d') \\
= \sup_{ \begin{scriptsize}
 \begin{array}{c}
d, d' \in \mathcal{P} \\
d \neq d'
\end{array}
\end{scriptsize}} K(d,d') \cdot \sin \theta(X,d,d')
\end{multline*}
as desired.
\end{proof}
We end this section with the following computation of $K(0,\infty)$ and $K(0,\Phi)$, for later use.
\begin{Prop}\label{Examples_K}
The following hold:
\begin{enumerate}[label=(\roman*)]
\item $K(0,\infty)= \frac{1}{l(\alpha_m)^2}$ and $K(0,\Phi) = \frac{1}{\Phi}\cdot K(0,\infty)$.
\item $\forall (d,d') \notin \Gamma_n \cdot (0,\infty)$, $K(d,d')\leq K(0,\Phi)$.
\end{enumerate}
\end{Prop}
\begin{proof}[Proof of Proposition~\ref{Examples_K}]
$(i)$ We use the notations of Figure \ref{heptagon}. The directions $d=0$ and $d'=\infty$ correspond to the vertical and the horizontal, respectively.
By definition
$$K(0,\infty) = \sup_{
 \begin{scriptsize}
 \begin{array}{c}
\alpha\subset S_0\ {\mathrm horizontal } \\
\beta\subset S_0\ {\mathrm vertical }
\end{array}
\end{scriptsize}
}
\frac{\mathrm{Int} (\alpha,\beta)}{l(\alpha)l(\beta)} =\sup_{i,j=1,\dots,m} \frac{\mathrm{Int} (\alpha_i,\beta_j)}{l(\alpha_i)l(\beta_j)}
$$
A quick look at the intersections shows that $\mathrm{Int} (\alpha_i,\beta_j) \in \{0,\pm 1\}$.  Moreover, $l(\alpha_k)=l(\beta_k)=\sin 2k\pi/(2m+1) \geq \sin \pi/(2m+1)$ and equality is realized for $k=m$. Since $\mathrm{Int} (\alpha_m,\beta_m) \neq 0$, we draw $K(0,\infty)=\cfrac{1}{l(\alpha_m)l(\beta_m)}=\cfrac{1}{l(\alpha_m)^2}$.\newline

The discussion for the directions $d=0$ and $d'=\Phi$ is similar. They correspond to the vertical and the direction of the diagonal of horizontal cylinders. It is clear that 
$$K(0,\Phi) = \sup_{
 \begin{scriptsize}
 \begin{array}{c}
\alpha\subset S_0\ {\mathrm vertical } \\
\beta\subset S_0\ {\mathrm diagonal \text{ } of \text{ } a \text{ } horizontal \text{ } cylinder } \\ 

\end{array}
\end{scriptsize}
}
\frac{\mathrm{Int} (\alpha,\beta)}{l(\alpha)l(\beta)}
$$
is maximal for $\alpha = \beta_m$ and $\beta = \alpha_1 + \beta_m$. And we have $K(0,\Phi) =\frac{\mathrm{Int}(\alpha,\beta)}{\alpha \wedge \beta} = \frac{1}{l(\alpha_1)l(\beta_m)} = \frac{1}{\Phi}K(0, \infty)$.\newline

$(ii)$ Let $(d,d') \notin \Gamma \cdot (0,\infty)$. Since $K(d,d')$ is invariant under the diagonal action of the Veech group and $d$ is a periodic direction, we can assume $d=\infty$, and, up to a horizontal shear, $d' \in ]0,\Phi[$. Notice that given a geodesic $\beta$ in direction $d'$, every intersection with any of the $\alpha_i$ requires a vertical length $l(\alpha_1)$ (this is where we use the fact that $d'$ is not vertical), that is
\begin{equation*}
\forall i \in \{1, \cdots, m \}, \text{ } l(\beta) \sin \theta(S_0,d,d') \geq l(\alpha_1) \mathrm{Int}(\alpha_i,\beta).
\end{equation*}
Hence
\begin{equation*}
\forall i \in \{1, \cdots, m \}, \text{ } \frac{\mathrm{Int} (\alpha_i,\beta)}{l(\alpha_i) l(\beta) \sin \theta(S_0,d,d')} \leq \frac{1}{l(\alpha_1)l(\alpha_i)}  
\end{equation*}
But $l(\alpha_i) l(\beta) \sin \theta(S_0,d,d') = \alpha_i \wedge \beta$, and $l(\alpha_i) \geq l(\alpha_m)$, so that the last equation reduces to
\begin{equation*}
\forall i \in \{1, \cdots, m \}, \text{ } \frac{\mathrm{Int} (\alpha_i,\beta)}{\alpha_i \wedge \beta} \leq \frac{1}{ l(\alpha_1)l(\alpha_m)} = \frac{1}{\Phi}K(0,\infty),
\end{equation*}
where the last equality follows from $(i)$. This concludes the proof of Proposition~\ref{Examples_K}.
\end{proof}

\section{Computation of $\KVol$ for the double $(2m+1)$-gon}\label{sec:2m+1-gon}
In this section we show that $\KVol(X_0)$ is realised by pairs of sides of the double $(2m+1)$-gon:
\begin{Prop}\label{Double2m+1gon}
For every pair of saddle connections $\alpha$ and $\beta$ on $X_0$, we have:
\begin{equation*}
\frac{Int(\alpha,\beta)}{l(\alpha) l(\beta)} \leq \frac{1}{l_0^2}
\end{equation*}
where $l_0$ is the length of the side of the $(2m+1)$-gon.\newline
Moreover, equality is achieved if and only if $\alpha$ and $\beta$ are two distinct sides of the regular $(2m+1)$-gon.
\end{Prop}

In particular, since the directions $d=0$ and $d'=\infty$ represent sides of the double $(2m+1)$-gon, we deduce the following:

\begin{Cor}\label{Victoire_en _X0}
For $X=X_0$ the double $(2m+1)$-gon, we have:
\begin{equation*}
K(X_0)= K(0,\infty)\cdot \sin \theta(X_0,0,\infty)
\end{equation*}
\end{Cor}

The main idea of the proof of Proposition \ref{Double2m+1gon} is to subdivide both saddle connections $\alpha$ and $\beta$ into segments of length at least $l_0$ such that each segment of $\alpha$ intersect each segment of $\beta$ at most once. While, strictly speaking,  we do not achieve that, we still get estimates good enough for our purpose, see Lemma  \ref{intersections_segments}. \newline

\subsection{Sectors and transition diagrams.}
Let $\alpha$ and $\beta$ be two saddle connections on the double $(2m+1)$-gon. We partition the set of possible directions into $2m+1$ sectors of angle $\frac{\pi}{2m+1}$, as in Figure \ref{sectors_heptagon} for the double heptagon.\newline

\begin{figure}[h]
\center
\includegraphics[height = 3.5cm]{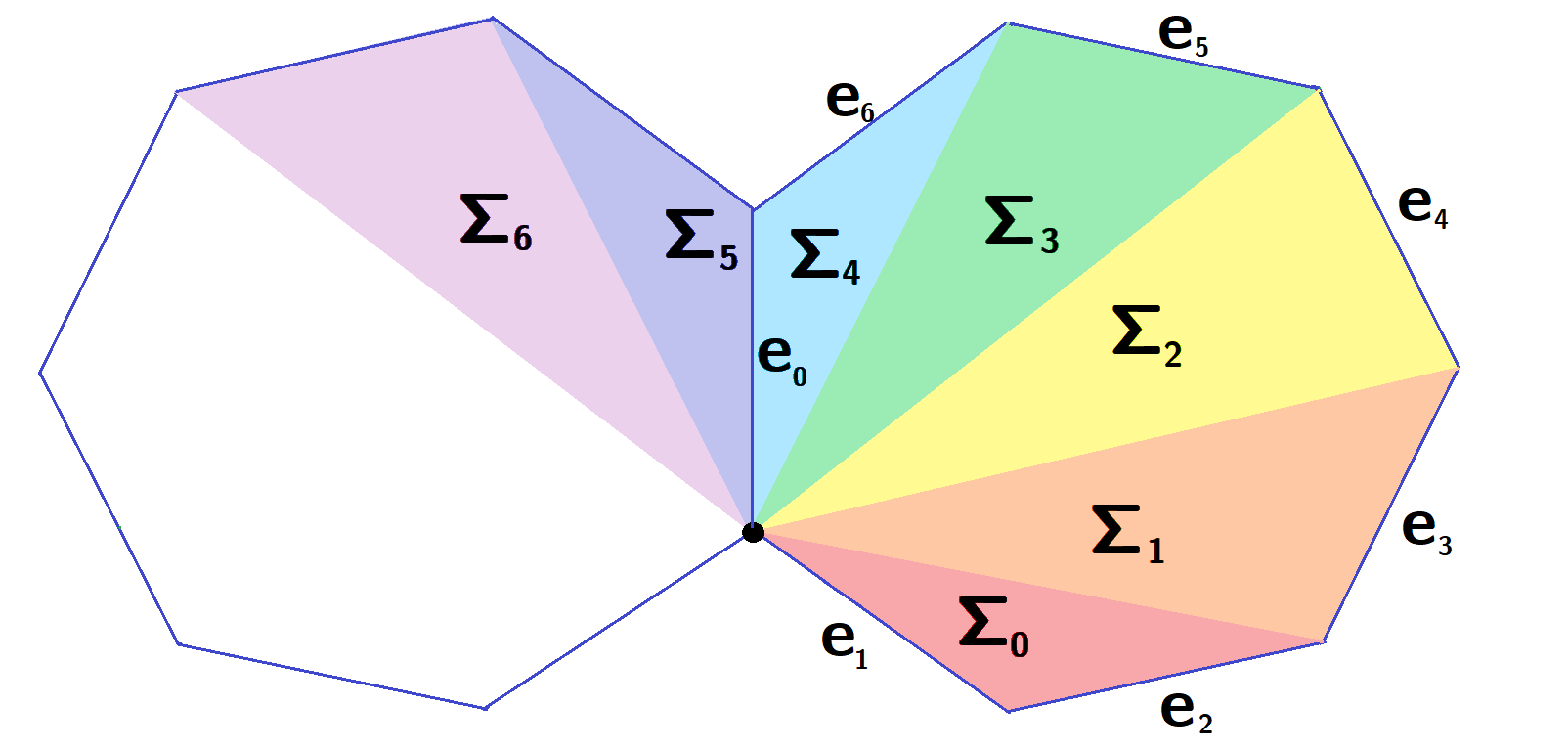}
\caption{The seven sectors for the double heptagon.}
\label{sectors_heptagon}
\end{figure}

To each sector $\Sigma_i$ there is  associated a \textit{transition diagram} which encodes  the admissible sequences of intersections with the sides of the double $(2m+1)$-gon, as in \cite[\S 3]{DL18}. Such a diagram looks like:
\begin{equation*}
e_{\sigma_i(1)} \leftrightharpoons e_{\sigma_i(2)} \leftrightharpoons \cdots \leftrightharpoons e_{\sigma_i(2n)}\leftrightharpoons e_{\sigma_i(2m+1)}
\end{equation*}
with $\sigma_i \in \mathfrak{S}_{2m+1}$, and it means that for any curve $\alpha$ in sector $\Sigma_i$, each intersection of $\alpha$ with $e_{\sigma_i(j)}$ is preceded and followed by an intersection with either $e_{\sigma_i(j-1)}$ or $e_{\sigma_i(j+1)}$ \footnote{Unless it reaches a singularity.}. In particular, each intersection with $e_{\sigma_i(1)}$ (resp. $e_{\sigma_i(2m+1)}$) is preceded (and followed) by an intersection with $e_{\sigma_i(2)}$ (resp. $e_{\sigma_i(2m)}$). We say that the side $e_{\sigma_i(1)}$ (resp. $e_{\sigma_i(2m+1)}$) is \textit{sandwiched} by $e_{\sigma_i(2)}$ (resp. $e_{\sigma_i(2m)}$) in the  sector $\Sigma_i$. 

For the sake of completeness, we provide the seven possible transition diagrams for the double heptagon.
\begin{equation}
\label{eq:sectors}
\begin{array}{ll}
\Sigma_0 : e_1 \leftrightharpoons e_2 \leftrightharpoons e_0 \leftrightharpoons e_3 \leftrightharpoons e_6 \leftrightharpoons e_4 \leftrightharpoons e_5 \\
\Sigma_1 : e_5 \leftrightharpoons e_6 \leftrightharpoons e_4 \leftrightharpoons e_0 \leftrightharpoons e_3 \leftrightharpoons e_1 \leftrightharpoons e_2 \\
\Sigma_2 : e_2 \leftrightharpoons e_3 \leftrightharpoons e_1 \leftrightharpoons e_4 \leftrightharpoons e_0 \leftrightharpoons e_5 \leftrightharpoons e_6 \\
\Sigma_3 : e_6 \leftrightharpoons e_0 \leftrightharpoons e_5 \leftrightharpoons e_1 \leftrightharpoons e_4 \leftrightharpoons e_2 \leftrightharpoons e_3 \\
\Sigma_4 : e_3 \leftrightharpoons e_4 \leftrightharpoons e_2 \leftrightharpoons e_5 \leftrightharpoons e_1 \leftrightharpoons e_6 \leftrightharpoons e_0 \\
\Sigma_5 : e_0 \leftrightharpoons e_1 \leftrightharpoons e_6 \leftrightharpoons e_2 \leftrightharpoons e_5 \leftrightharpoons e_3 \leftrightharpoons e_4 \\
\Sigma_6 : e_4 \leftrightharpoons e_5 \leftrightharpoons e_3 \leftrightharpoons e_6 \leftrightharpoons e_2 \leftrightharpoons e_0 \leftrightharpoons e_1 \\
\end{array}
\end{equation}

\subsection{Construction of the subdivision}
\label{sec:construction}

Let us denote by $\Sigma_\alpha$ (resp. $\Sigma_\beta$) the sector of $\alpha$ (resp. $\beta$), and $\sigma_\alpha$ (resp. $\sigma_\beta$) the corresponding permutation. Now, we cut $\alpha$ (resp. $\beta$) at each intersection with a non-sandwiched side in the sector $\Sigma_\alpha$ (resp. $\Sigma_\beta$). We get a decomposition $\alpha = \alpha_1 \cup \cdots \cup \alpha_k$ and $\beta = \beta_1 \cup \cdots \cup \beta_l$ with $k,l \geq 1$ and each segment is either (see Figure \ref{sandwich_example}): 
\begin{itemize}
\item A non-sandwiched segment which goes from one side to another non-adjacent side of one of the $(2m+1)$-gons.
\item A sandwiched segment which intersects a sandwiched side in its interior. Such segments go through both $(2m+1)$-gons (see Figure \ref{cases_proof}).
\item an initial or final segment $\alpha_1$, $\alpha_k$, $\beta_1$, $\beta_l$. We also call such segments non-sandwiched.
\end{itemize}

\begin{Nota}
When a segment $\alpha_i$ (or $\beta_j$) intersects the side $e$ which is sandwiched by $e'$ in the corresponding sector, we say that $\alpha_i$ is of type $e' \to e \to e'$. 
\end{Nota}
\begin{Rema}
\label{rem:parallelogram}
For a sandwiched segment $\alpha_i$ of type $e' \to e \to e'$, there is a parallelogram $P(e',e)\subset X_0$ with the following property: one of its sides is $e'$ and one of its diagonals is $e$. The segment $\alpha_i$ goes from one  $e'$-side to the opposite side. The closure of $P(e',e)$ is a cylinder.
\end{Rema}
Notice that for each sector $\Sigma_i$, the sides of the $(2m+1)$-gon which are sandwiched in $\Sigma_i$ are those having direction in the boundary of $\Sigma_i$. For instance, the sides of the double heptagon which are sandwiched in the sector $\Sigma_0$ are $e_1$ and $e_5$ (see Figure \ref{sectors_heptagon}). 

Moreover, the side of the $(2m+1)$-gon with direction in $\Sigma_i \cap \Sigma_{i-1}$ is sandwiched in both sectors $\Sigma_i$ and $\Sigma_{i-1}$, but in $\Sigma_i$ it is sandwiched by its successor in the cyclic order (modulo $2m+1$), while in $\Sigma_{i-1}$ it is sandwiched by its predecessor. For instance, $e_1$ is sandwiched by $e_2$ in $\Sigma_0$, while it is sandwiched by $e_0$ in $\Sigma_{2m}$ ($\Sigma_{6}$ for the double heptagon).\newline

Since no two sectors have the same pair (sandwiched side, sandwiching side), 
prescribing the type of $\alpha_i$ automatically tells which sector the direction of $\alpha$ belongs to.
\begin{Rema}
If $\alpha$ is a diagonal of the double $(2m+1)$-gon, the sector is not uniquely defined. However, in such cases $\alpha$ is not divided into pieces and $k =1$.
\end{Rema}

\begin{figure}[h]
\center
\includegraphics[height = 5cm]{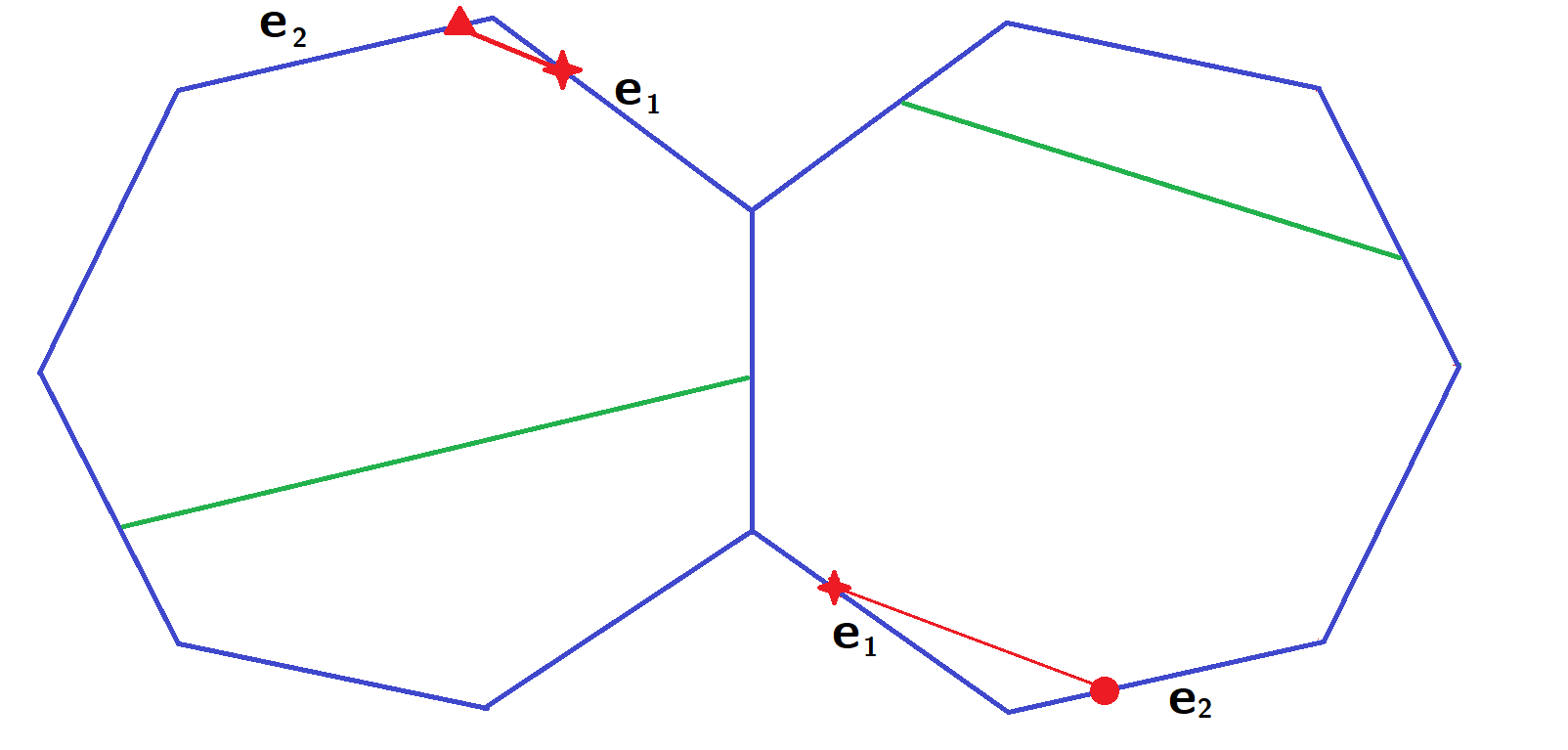}
\caption{Examples of non-sandwiched segments in green and a sandwiched segment in red. The sandwiched segment is of type $e_2 \to e_1 \to e_2$. Remark that the points $\bigcirc$ and $\triangle$ on the side $e_2$ are not the same.}
\label{sandwich_example}
\end{figure}

The next two  lemmas are the reason for this peculiar way of subdivising saddle connections. 

\begin{Lem}\label{lenght_segments}
Every segment of $\alpha$ (resp. $\beta$) has length at least $l_0$, with equality if and only if $\alpha$ (resp. $\beta$) is a side of the double $(2m+1)$-gon.
\end{Lem}
\begin{proof}
A non-sandwiched segment goes from one side (or vertex) of a $(2m+1)$-gon to a {\em non-adjacent} side in the same $(2m+1)$-gon. So its length is a least  $l_0$.

Now take a sandwiched segment $\alpha_i$ which intersects a sandwiched side $e$ in its interior. If the type of $\alpha_i$ is $e' \to e \to e'$, then by Remark~\ref{rem:parallelogram}, $\alpha_i$ goes from the side $e'$ of $P(e',e)$ to the opposite side of $P(e',e)$. In particular the length of $\alpha_i$ is no less than that of $e'$ which is $l_0$.
\end{proof}
\subsection{Study of the intersections}
In this section, we investigate the possible intersections between the segments of $\alpha$ and $\beta$. First observe that $\alpha_i$ and $\beta_j$ can have nontrivial intersections on the interior of the sides of a sector $\Sigma_\alpha$. However if this happen, we can slightly deform $\alpha$ as follows. We change the slope of $\alpha_i$ and $\alpha_{i+1}$ so that the new segment $\alpha'_i$ intersects $\beta_j$ in the interior of $\Sigma_\alpha$ the same number of times $\alpha_i$ intersects $\beta_j$. We choose the deformation small enough so that we do not create new intersections with the others segments. The new path $\alpha' = \alpha_1 \cup \cdots  \cup \alpha'_i \cup \alpha'_{i+1} \cup \cdots  \cup \alpha_k$ is homologous to $\alpha$ by construction. In the sequel we will simply use $\alpha$ instead of $\alpha'$.

Now, since $\beta$ is made of segments of straight lines in the same direction, and $\alpha$ is made of segments whose directions are close to a given direction, all non-singular intersections have the same sign. In particular, adding the possible singular intersection, it gives:
\begin{equation}\label{eq_intersections}
Int(\alpha,\beta) \leq \sum_{i,j} |Int(\alpha_i, \beta_j)| + 1
\end{equation}
where $|Int(\alpha_i, \beta_j)|$ is the geometric intersection between $\alpha_i$ and $\beta_j$. 

\begin{Lem}\label{intersections_segments}
If $\alpha$ and $\beta$ are not both diagonals, then $\text{Int}(\alpha, \beta) \leq kl$.
\end{Lem}
\begin{proof}[Proof of Lemma~\ref{intersections_segments}]
We will show that $\sum_{i,j} |Int(\alpha_i, \beta_j)| \leq kl-1$. Let us fix $i,j$. We first observe that if either $\alpha_i$ or $\beta_j$ is a non-sandwiched segment, then $\alpha_i$ and $\beta_j$ intersect at most once (possibly on a side). Indeed a non-sandwiched segment goes from one side to another non-adjacent side of one of the $(2m+1)$-gons. In particular it is a segment that is contained entirely in one of the $(2m+1)$-gons. A sandwiched segment consists of two straight lines, not contained in the same $(2m+1)$-gon. Hence in total they intersect at most once.

Thus it remains to consider the case where $\alpha_i$ and $\beta_j$ are sandwiched segments. Up to a rotation and a symmetry, we can assume $\alpha_i$ is of type $e_2 \to e_1 \to e_2$ (see~Figure \ref{sectors_heptagon}). The sector determined by $\alpha$ is necessarily $\Sigma_0$.

Now if $\beta_j$ is sandwiched but neither $e_1$ nor $e_2$ appear in the type of $\beta_j$, then $\alpha_j$ is contained in the parallelogram $P(e_k,e_l)$ (defined in Remark \ref{rem:parallelogram}) for some $e_k,e_l\not \in \{e_1,e_2\}$. In particular $\alpha_i$ does not intersect this parallelogram, and so not $\beta_j$ either. \newline

Eventually it remains to treat the following cases where $\beta_j$ is of type:
$$
\begin{array}{llllllll}
(1)& e_0 \to e_1 \to e_0 &&&&& (2)& e_1 \to e_0 \to e_1 \\
(3)& e_1 \to e_2 \to e_1 &&&&& (4)& e_2 \to e_1 \to e_2 \\
(5)& e_2 \to e_3 \to e_2 &&&&& (6)& e_3 \to e_2 \to e_3
\end{array}
$$

In all situations but (3), we can show that $\alpha_i$ and $\beta_j$ intersects at most once. We proceed case by case. \newline

\begin{figure}[h]
\center
\includegraphics[height = 13cm]{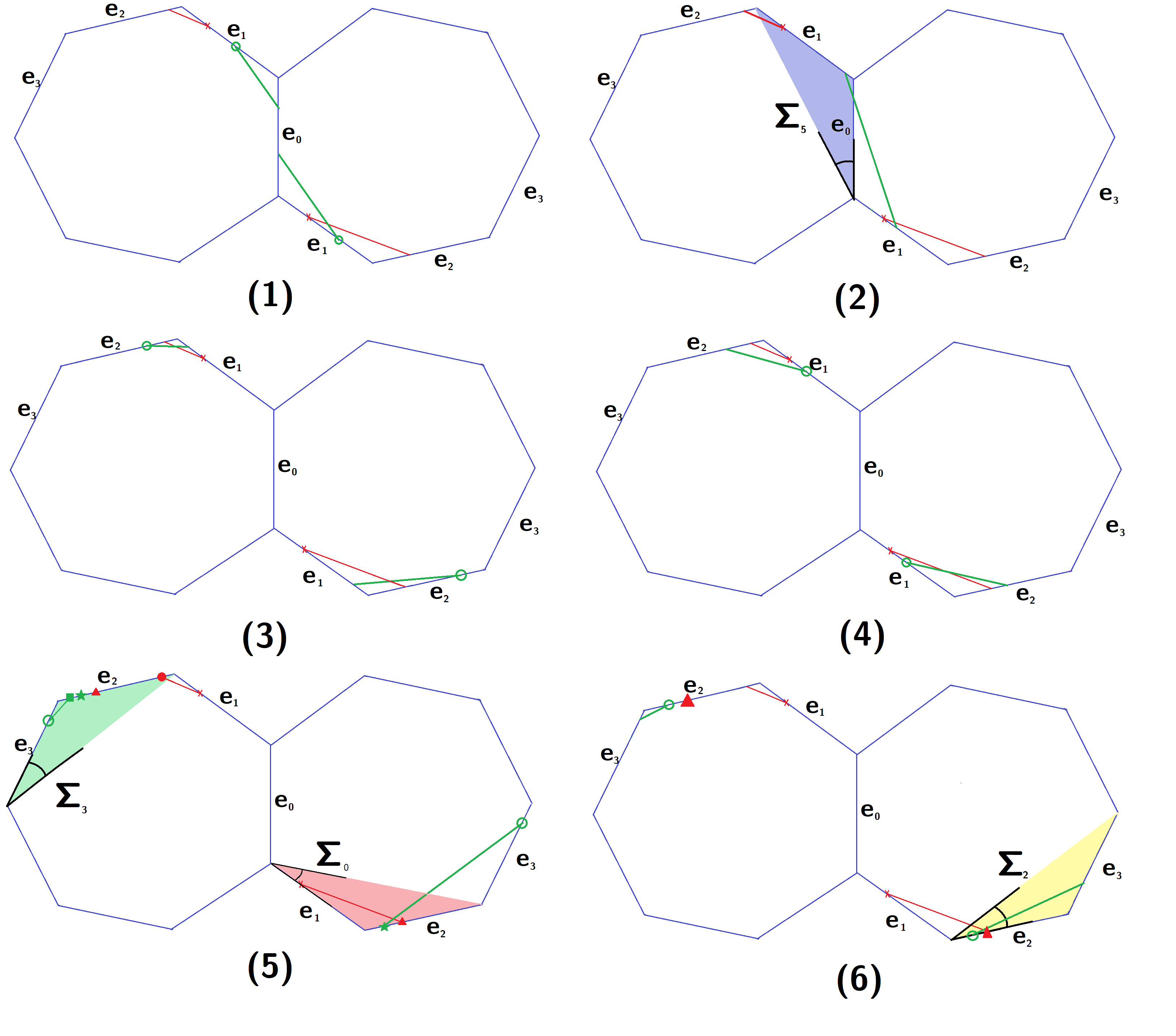}
\caption{The six different cases where $\alpha_i$ (in red) and $\beta_j$ (in green) are sandwiched and intersect.}
\label{cases_proof}
\end{figure}

\textbf{Case (1):} Recall that $\alpha_i$ is contained in the parallelogram $P(e_2,e_1)$ and goes from one $e_2$-side  to the opposite side. Let us denote by $e,e'$ the two other sides of $P(e_2,e_1)$. Similarly $\beta_j$ is contained in another parallelogram $P(e_0,e_1)$ sharing the same diagonal, and $\beta_j$ goes from one side $e_0$ to the opposite side. We can see that the intersection of $P(e_2,e_1)$ with $P(e_0,e_1)$ is connected: it is a parallelogram. In particular the intersection of $\beta_j$ with $P(e_2,e_1)$ goes from the side $e$ to the side $e'$ and thus intersects $\alpha_i$ exactly once (see~Figure~\ref{fig:lemme6.7cas1}). \newline

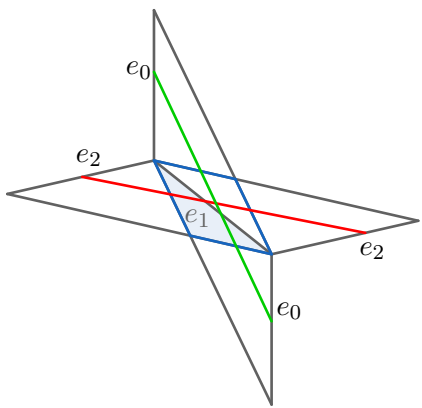
\begin{figure} [h]
	\begin{center}
\definecolor{qqccqq}{rgb}{0,0.8,0}
\definecolor{ffqqqq}{rgb}{1,0,0}
\definecolor{wrwrwr}{rgb}{0.3803921568627451,0.3803921568627451,0.3803921568627451}
\definecolor{rvwvcq}{rgb}{0.08235294117647059,0.396078431372549,0.7529411764705882}
\begin{tikzpicture}[scale=2,line cap=round,line join=round,>=triangle 45,x=1cm,y=1cm]

\fill[line width=2pt,color=rvwvcq,fill=rvwvcq,fill opacity=0.10000000149011612] (0,0) --  (0.7818314824680296,-0.6234898018587334) -- (0.24078730940376422,-0.5) -- cycle;
\draw [line width=1pt,color=wrwrwr] (-0.9749279121818234,-0.22252093395631434)-- (0,0);
\draw [line width=1pt,color=wrwrwr] (0,0)-- (1.756759394649853,-0.40096886790241915);
\draw [line width=1pt,color=wrwrwr] (-0.9749279121818234,-0.22252093395631434)-- (0.7818314824680296,-0.6234898018587334);
\draw [line width=1pt,color=wrwrwr] (0.7818314824680296,-0.6234898018587334)-- (1.756759394649853,-0.40096886790241915);
\draw [line width=1pt,color=wrwrwr] (0,0)-- (0.7818314824680296,-0.6234898018587334);
\draw [line width=1pt,color=wrwrwr] (0,1)-- (0,0);
\draw [line width=1pt,color=wrwrwr] (0,0)-- (0.7818314824680296,-1.6234898018587334);
\draw [line width=1pt,color=wrwrwr] (0.7818314824680296,-1.6234898018587334)-- (0.7818314824680296,-0.6234898018587334);
\draw [line width=1pt,color=wrwrwr] (0.7818314824680296,-0.6234898018587334)-- (0,1);
\draw [line width=1pt,color=rvwvcq] (0,0)-- (0.5410441730642654,-0.12348980185873357);
\draw [line width=1pt,color=rvwvcq] (0.5410441730642654,-0.12348980185873357)-- (0.7818314824680296,-0.6234898018587334);
\draw [line width=1pt,color=rvwvcq] (0.7818314824680296,-0.6234898018587334)-- (0.24078730940376422,-0.5);
\draw [line width=1pt,color=rvwvcq] (0.24078730940376422,-0.5)-- (0,0);
\draw [line width=1pt,color=ffqqqq] (-0.4774446222089946,-0.10897361940187346)-- (1.4045917409084931,-0.48134883676017437);
\draw [line width=1pt,color=qqccqq] (0,0.5896296296296294)-- (0.7818314824680296,-1.0696296296296293);

\draw[color=wrwrwr] (0.2903703703703685,-0.38518518518518463) node {$e_1$};

\draw[color=black] (-0.43259259259259475,0.0177777777777781) node {$e_2$};

\draw[color=black] (1.4518518518518504,-0.6) node {$e_2$};

\draw[color=black] (-0.1,0.6) node {$e_0$};

\draw[color=black] (.9,-1) node {$e_0$};

\end{tikzpicture}
\caption{Case 1 re-drawn}
\label{fig:lemme6.7cas1}
\end{center}
\end{figure}

\textbf{Case (4):} By Remark~\ref{rem:parallelogram} $\alpha_i$ and $\beta_j$ are contained in the {\em same} $P(e_2,e_1)$. They both go from one side $e_2$ to the opposite. In particular they intersect at most once. \newline

\textbf{Case (2):}
By~\eqref{eq:sectors}, since $e_0$ is sandwiched by $e_1$, the direction of $\beta$ lies in the sector $\Sigma_5$. Moreover $\beta_j$ lies in the parallelogram $P(e_1,e_0)$ and goes from the side $e_1$ to the opposite (see Figure~\ref{fig:case2}). The segments $\alpha_i$ and $\beta_j$ intersect each other at most twice. Assume by contradiction $\alpha_i \cap \beta_j =\{p,q\}$ with $p\neq q \in X_0$. Since $p$ and $q$ belong to different parts of the segment $\alpha_i$, they belong to  different copies of the $(2m+1)$-gon. The slope $s$ of the holonomy vector defined by $pq$ coincide with the slope of $\beta$ and thus belongs to the sector $\Sigma_5$ {\em i.e.} $\mathrm{slope}(e_4) \leq s \leq 0$ (with equality iff $\beta$ is a diagonal). On the other hand, the intersection of $\alpha_i$ with the two sides $e_1$ on the plane template of Figure~\ref{fig:case2} determines two points $c$ and $d$. 

By construction the slope $s$ satisfies $s \leq \mathrm{slope}(cd)$. Since $\mathrm{slope}(cd)=\mathrm{slope}(e_4)$ we get that $\beta$ is a diagonal. We run into a contradiction because $\beta_j=\beta$ is not sandwiched (see beginning of Section~\ref{sec:construction}), and therefore $\mathrm{Int}(\alpha_i,\beta_j) \leq 1$. \newline

\begin{figure}[h]

		\begin{center}
		\begin{tikzpicture}[scale=2]
		\def\tr{(-2,0)}
		\draw[thick] (1.0000, 0.00000) -- (0.62349,0.78183) -- (-0.22252,0.97493) -- (-0.90097,0.43388) -- (-0.90097,-0.43388) -- (-0.22252,-0.97493) -- (0.62349, -0.78183) -- cycle;
		
		\draw[thick] (-1.0000-1.80194, 0.00000) -- (-0.62349-1.80194,0.78183) -- (0.22252-1.80194,0.97493) -- (0.90097-1.80194,0.43388) -- (0.90097-1.80194,-0.43388) -- (0.22252-1.80194,-0.97493) -- (-0.62349-1.80194, -0.78183) -- cycle;

		\draw[thick,dashed] (-0.90097,0.43388) -- (-0.22252,-0.97493);
		\draw[thick,dashed] (-0.90097,-0.43388) -- (0.62349, -0.78183);
		\draw[thick,dashed] (0.90097-1.80194,-0.43388) -- (0.22252-1.80194,0.97493);
		\draw[thick,dashed] (-0.62349-1.80194,0.78183) -- (0.90097-1.80194,0.43388);
		
		\coordinate (A) at (intersection of -0.90097,0.43388-- -0.22252,-0.97493 and -0.90097,-0.43388 -- 0.62349, -0.78183);
		\draw[draw=none,fill=gray!50] (A) -- (-0.22252,-0.97493) -- (-0.90097,-0.43388) -- cycle;
		\coordinate (B) at (intersection of 0.90097-1.80194,-0.43388-- 0.22252-1.80194,0.97493 and -0.62349-1.80194,0.78183 -- 0.90097-1.80194,0.43388);
		\draw[draw=none,fill=gray!50] (B) -- (-0.90097,0.43388) -- (0.22252-1.80194,0.97493) -- cycle;
		
		\draw[thick,color=green] (-1.4,0.83) -- (-0.32,-0.9);
		\draw[thick,color=red] (-1.87,0.9) -- (-1.2,0.67);
		\draw[thick,color=red] (-0.44,-0.8) -- (0,-0.93);
		
		\coordinate (p) at (intersection of -1.4,0.83-- -0.32,-0.9 and -1.87,0.9 -- -1.2,0.67);
		\coordinate (q) at (intersection of -1.4,0.83-- -0.32,-0.9 and -0.44,-0.8 -- 0,-0.93);
		
		\coordinate (c) at (intersection of -0.44,-0.8 -- 0,-0.93,-0.9 and -0.22252,-0.97493 -- -0.90097,-0.43388);
		\coordinate (d) at (intersection of -1.87,0.9 -- -1.2,0.67 and -0.90097,0.43388 -- 0.22252-1.80194,0.97493);

		\filldraw[black] (p) circle (0.2pt) node[left]{$\scriptstyle{p}$};
		\filldraw[black] (q) circle (0.2pt) node[right]{$\scriptstyle{q}$};

		  \draw (c) circle[radius=0.4pt] node[left]{$\scriptstyle{c}$};		
  		  \draw (d) circle[radius=0.4pt] node[right]{$\scriptstyle{d}$};

		\draw (-.85,0) node[left]{$\scriptstyle e_0$};
		\draw (-0.65,.65) node[above]{$\scriptstyle e_6$};
		\draw (-0.74,-0.85) node[above]{$\scriptstyle e_1$};
		\draw (.8,0.5) node[right]{$\scriptstyle e_4$};
		\draw (-0.67+.8,1) node[right]{$\scriptstyle e_5$};
		\draw (0.2,-0.9) node[below]{$\scriptstyle e_2$};
		\draw (-0.74+1.7,-0.5) node[above]{$\scriptstyle e_3$};
		\end{tikzpicture}
	\end{center}
	\caption{The parallelogram $P(e_1,e_0)$ and two intersections in Case (2): $\beta\not \in \Sigma_5$}
	\label{fig:case2}
\end{figure}
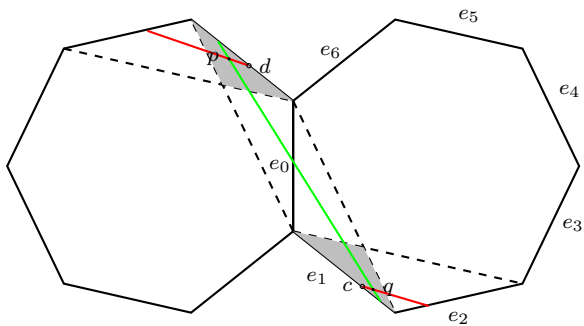

\textbf{Case (6):}
This case is the same as Case (2)  rotating by an angle $\frac{2\pi}{2m+1}$ and swapping $\alpha_i$ and $\beta_j$.\newline

\textbf{Case (5):}
By~\eqref{eq:sectors}, since $e_3$ is sandwiched by $e_2$, the direction of $\beta$ lies in the sector $ \Sigma_3$ and $\beta$ is contained in the parallelogram $P(e_2,e_3)$. In particular its slope verifies $s\geq \mathrm{slope}(e_6)$. The segment $\alpha_i$ intersects  $P(e_2,e_3)$ with two connected components. One such component intersects one side $e_2$ at a point $a$ while the other component intersects the other side $e_2$ at a point $b$ (which project in the double $(2m+1)$-gon to the red bullet and the red triangle in Figure~\ref{cases_proof} Case (5)).
As in Case (2) the segments $\alpha_i$ and $\beta_j$ intersect each other at most twice. Assume by contradiction $\alpha_i \cap \beta_j =\{p,q\}$ with $p\neq q \in P(e_2,e_3)$. Necessarily $\mathrm{slope}(ab) \geq \mathrm{slope}(pq)=s$. Since $\alpha\in \Sigma_0$ we can check that $\mathrm{slope}(ab)\leq \mathrm{slope}(e_6)$ (recall that the slope of $\alpha'_i$ is very close to the slope of $\alpha_i$).

Thus $s=\mathrm{slope}(e_6)$ and $\beta$ is a diagonal: we again run into a contradiction because $\beta_j=\beta$ is not sandwiched. Therefore $\mathrm{Int}(\alpha_i,\beta_j) \leq 1$ as desired. \newline

\textbf{Conclusion:}
Setting aside case (3) which will be considered below, we have for every $i,j$, $|Int(\alpha_i,\beta_j)| \leq 1$. In particular $\sum_{i,j} |Int(\alpha_i, \beta_j)| \leq kl$. Recall that we want the left quantity to be less than $kl-1$ instead of $kl$; the desired bound comes from the following observation:
 up to permuting $\alpha$ and $\beta$, we may assume $\alpha$ is not a diagonal. Then $\alpha_1$ is non-sandwiched and lies in one of the $(2m+1)$-gons while the second non-sandwiched $\alpha_i$ lies in the other $(2m+1)$-gon\footnote{Notice that the last non-sandwiched segment before a sequence of sandwiched segments and the next non-sandwiched segment after such a sequence lie in different $(2m+1)$-gons, as in Figures \ref{case_3_image_2} and \ref{case_3_image_3}.}. In particular, since $\beta_1$ is non sandwiched, it lies in one of the two $(2m+1)$-gons and it cannot intersect both $\alpha_1$ and the next non-sandwiched $\alpha_i$.

In particular, $\sum_{i,j} |Int(\alpha_i, \beta_j)| \leq kl-1$. Hence, by Formula \ref{eq_intersections}, we get that $Int(\alpha,\beta) \leq kl$.

\subsection*{Treating case (3)}

In this paragraph we assume there are indices $i,j$ such that $\alpha_i$ is sandwiched of type $e_2 \to e_1 \to e_2$ and $\beta_j$ is sandwiched of type $e_1 \to e_2 \to e_1$. In this case, $\alpha_i$ and $\beta_j$ could intersect twice, but we will show that if this happens then there is an index $j'$ such that $\alpha_i$ and $\beta_{j'}$ don't intersect. Hence the conclusion $\sum_{i,j} |Int(\alpha_i,\beta_j)| \leq kl$ will still hold (and in fact the inequality will be strict, as we require).\newline

\begin{figure}[h]
\center
\includegraphics[height = 7cm]{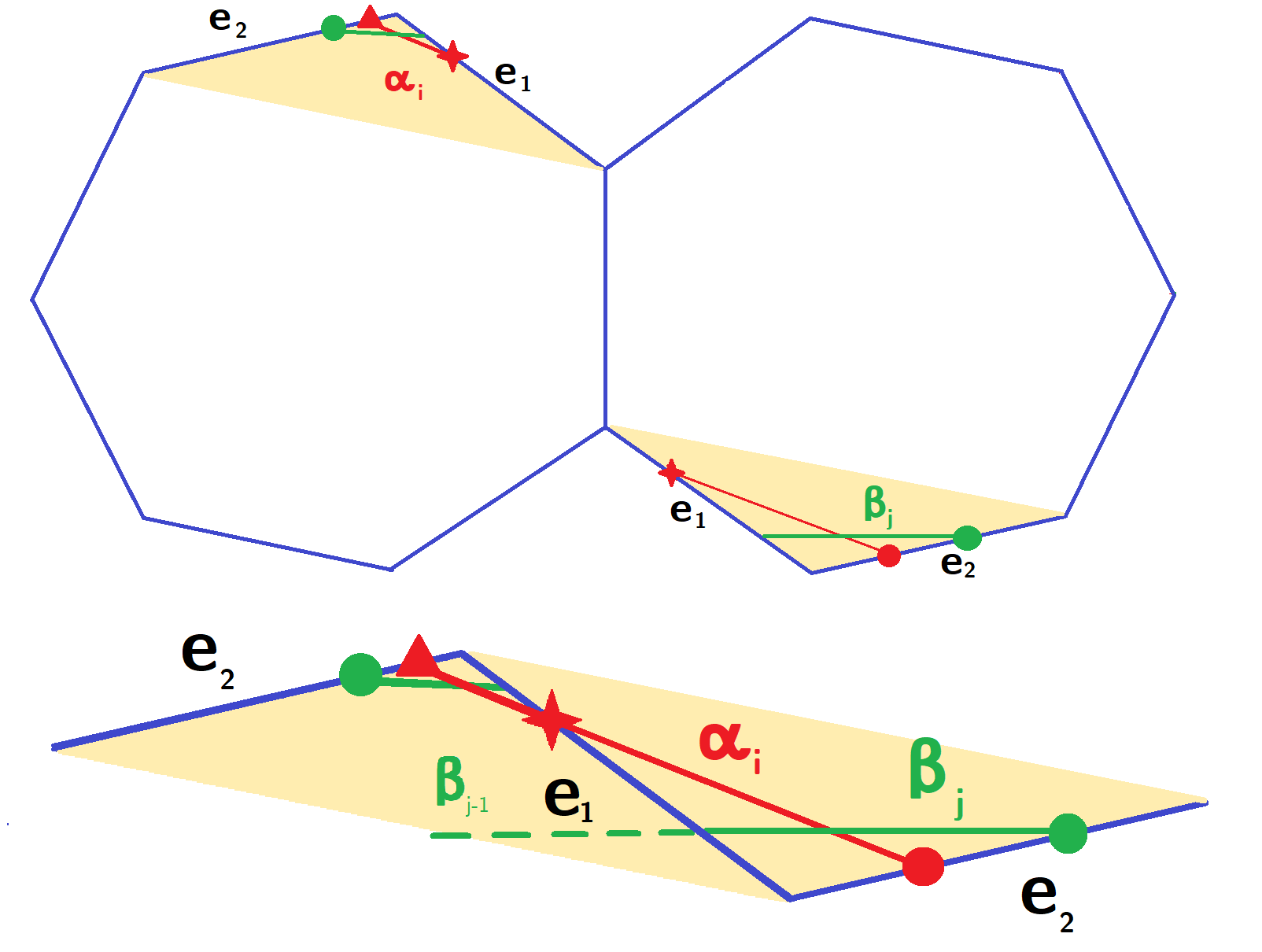}
\caption{$\alpha_i$ and $\beta_j$ could intersect twice in the configuration of case $3$. Below, a closer look at the cylinder $P(e_2,e_1)$. In the example of this picture, $\alpha_i$ does not intersect $\beta_{j-1}$.}
\label{case_3}
\end{figure}

To do that let us assume $\alpha_{i_0}, \dots, \alpha_{i_0 +p}$ (resp. $\beta_{j_0}, \dots, \beta_{j_0 +q}$) are consecutively sandwiched of type $e_2 \to e_1 \to e_2$ (resp. $e_1 \to e_2 \to e_1$), this sequence being maximal (i.e $\alpha_{i_0-1}$ and $\alpha_{i_0+p+1}$ -- resp. $\beta_{j_0-1}$ and $\beta_{j_0+q+1}$ -- are not sandwiched). An example of such a configuration is depicted in Figure \ref{case_3_image_2} for $p=q=0$ and in Figure \ref{case_3_image_3} for $p=q=3$. We claim that there are at most $(p+3)(q+2)$ intersections between $\alpha_{i_0-1} \cup \alpha_{i_0} \cup \cdots \cup \alpha_{i_0+p} \cup \alpha_{i_0+p+1}$ and $\beta_{j_0-1} \cup \beta_{j_0} \cup \cdots \cup \beta_{j_0+q} \cup \beta_{j_0+q+1}$, while there are $(p+3)(q+3)$ pairs of segments.\newline

\begin{figure}[h]
\center
\includegraphics[height = 3cm]{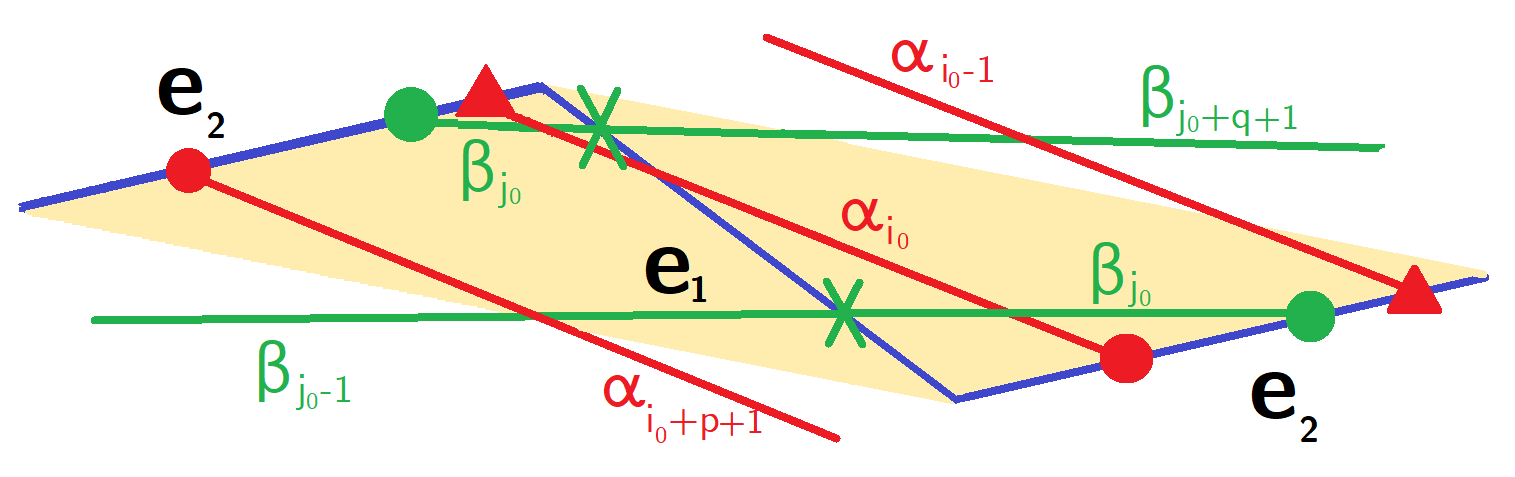}
\caption{$\alpha_{i_0-1} \cup \alpha_{i_0} \cup \cdots \cup \alpha_{i_0+p} \cup \alpha_{i_0+p+1}$ and $\beta_{j_0-1} \cup \beta_{j_0-1} \cup \cdots \cup \beta_{j_0+q} \cup \beta_{j_0+q+1}$ for $p=q=0$. There are only six intersections but nine pairs of segments.}
\label{case_3_image_2}
\end{figure}

\begin{figure}[h]
\center
\includegraphics[height = 3cm]{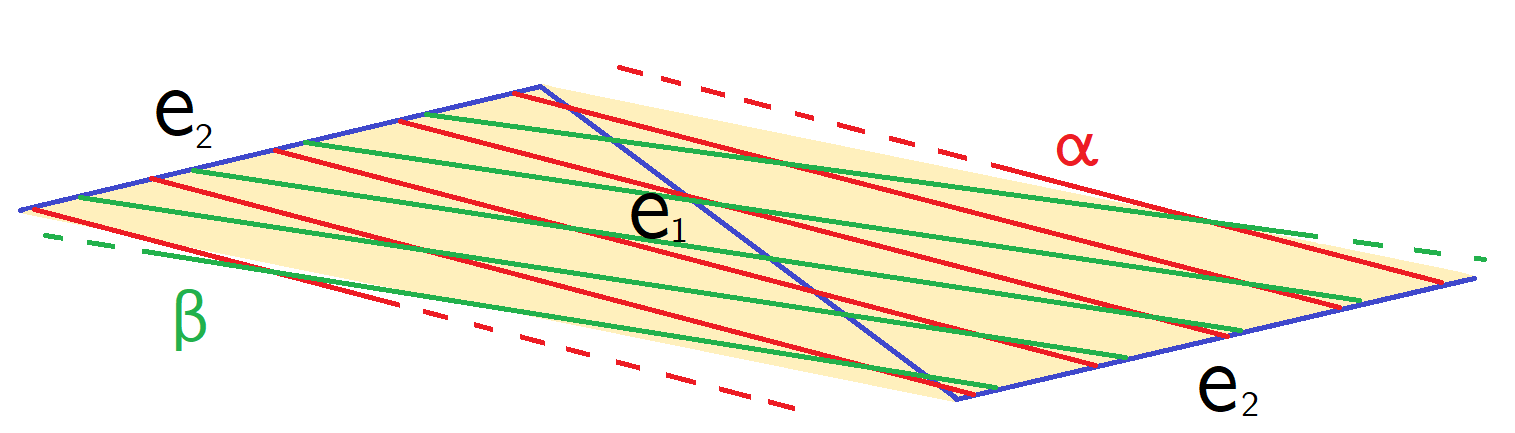}
\caption{$\alpha_{i_0-1} \cup \alpha_{i_0} \cup \cdots \cup \alpha_{i_0+p} \cup \alpha_{i_0+p+1}$ and $\beta_{j_0-1} \cup \beta_{j_0-1} \cup \cdots \cup \beta_{j_0+q} \cup \beta_{j_0+q+1}$ for $p=q=3$. There are only ten intersections but thirty-six pairs of segments.}
\label{case_3_image_3}
\end{figure}

Indeed, in this configuration $\alpha_{i_0-1} \cup \alpha_{i_0} \cup \cdots \cup \alpha_{i_0+p} \cup \alpha_{i_0+p+1}$ and $\beta_{j_0-1} \cup \beta_{j_0} \cup \cdots \cup \beta_{j_0+q} \cup \beta_{j_0+q+1}$ go through the cylinder $P(e_2,e_1)$ defined by $e_1$ and $e_2$, as in Figure \ref{case_3_image_2}. Now, instead of cutting $\beta$ each time it crosses $e_1$ we can cut $\beta$ each time it crosses $e_2$. Notice that $\beta$ crosses $e_1$ once  more than it crosses $e_2$, so it gives a decomposition $\beta_{j_0-1} \cup \cdots \cup \beta_{j_0+q} \cup \beta_{j_0+q+1} = \tilde{\beta}_{j_0} \cup \cdots \cup \tilde{\beta}_{j_0+q} \cup \tilde{\beta}_{j_0+q+1}$ with only $q+2$ segments while each $\tilde{\beta}_j$ for $j \in \llbracket j_0, j_0+q+1 \rrbracket$ can intersect each of the $\alpha_i$ for $i \in \llbracket i_0-1, i_0+p+1 \rrbracket$ at most once, which leaves at most $(p+3)(q+2)$ intersections.\newline

In conclusion, summing with all other segments yields $\sum_{i,j} |Int(\alpha_i, \beta_j)| < kl$. Adding the possible singular intersection, we get the desired result. This concludes the proof of Lemma \ref{intersections_segments} 
\end{proof}

\subsection{Conclusion}
We are now able to prove the main proposition of this section.

\begin{proof}[Proof of Proposition \ref{Double2m+1gon}]
If either $\alpha$ or $\beta$ is not a diagonal, then:
\begin{itemize}
\item $l(\alpha)l(\beta) > kl \cdot l_0^2$ by Lemma \ref{lenght_segments},
\item $Int(\alpha, \beta) \leq kl$ by Lemma \ref{intersections_segments},
\end{itemize}
In particular, we have:
\begin{equation*}
\frac{Int(\alpha,\beta)}{l(\alpha)l(\beta)} < \frac{1}{l_0^2}
\end{equation*}
as desired.\newline
Otherwise, both $\alpha$ and $\beta$ are diagonals. Then:
\begin{enumerate}
\item either none of them is a  side of a $(2m+1)$-gon and then:
\begin{enumerate}
\item $l(\alpha)l(\beta) \geq 4\cos^2(\frac{\pi}{2m+1})l_0^2 > 2l_0^2$ because the shortest diagonals of the $(2m+1)$-gon which are not sides have length $2\cos(\frac{\pi}{2m+1})l_0$.
\item $Int(\alpha,\beta) \leq 2$ because there is at most one non-singular intersection and one singular intersection.
\end{enumerate}
In particular, $\frac{Int(\alpha,\beta)}{l(\alpha)l(\beta)} < \frac{1}{l_0^2}$.
\item or $\alpha$ (up to a change in names) is a side, and then:
\begin{enumerate}
\item $Int(\alpha,\beta) \leq 1$ as there is no non-singular intersection,
\item $l(\alpha)l(\beta) \geq l_0^2$ with equality if and only if both $\alpha$ and $\beta$ are sides of a $(2m+1)$-gon.
\end{enumerate}
In particular, we have $\frac{Int(\alpha,\beta)}{l(\alpha)l(\beta)} \leq \frac{1}{l_0^2}$ with equality if and only if both $\alpha$ and $\beta$ are sides of a $(2m+1)$-gon.
\end{enumerate}
This concludes the proof of Proposition \ref{Double2m+1gon}.
\end{proof}


\section{Extension to the Teichm\"uller disc}\label{sec:extension}
In this section, we finally show our main result:
\begin{Theo}\label{main_result}
For any surface $X$ in the Teichmüller disc of the double $(2m+1)$-gon, we have:
\begin{equation}\label{Expression_KVol}
K(X) = K(0,\infty) \sin \theta (X,0,\infty).
\end{equation}
\end{Theo}
Theorem~\ref{thm:main:intro} follows directly as $\sin \theta(X,0, \infty) = \cfrac{1}{\cosh d_{hyp}(X, \gamma_{0,\infty})}$ by Proposition~\ref{prop:banane}. Before proving Theorem~\ref{main_result}, we show how to deduce Corollary~\ref{cor:main}.

\begin{proof}[Proof of Corollary~\ref{cor:main}]
Since $\Vol(S_0) = \frac{n}{2}\cos \frac{\pi}{n}$ by~\ref{eq:lengths} and the furthest point of $\mathcal D$ from $\gamma_{0,\infty}$ is $X_0$, the corresponding angle $\sin \theta (X_0,0,\infty)$ being equal to $\sin \frac{\pi}{n}$, Equation~\eqref{Expression_KVol} implies
$$
\frac{n}{2}  \cos \frac{\pi}{n} \cdot K(0,\infty)\cdot  \sin \frac{\pi}{n} \leq \KVol(X) \leq \frac{n}{2}  \cos \frac{\pi}{n} \cdot K(0,\infty).
$$
We conclude with Proposition~\ref{Examples_K} and Equation~\eqref{eq:lengths}: $K(0,\infty)= \frac{1}{l(\alpha_m)^2} = \frac{1}{\sin^2 \frac{\pi}{n}}$.

The maximum is achieved precisely when $\sin \theta (X,0,\infty)=1$ {\em i.e.} $X$ belongs to the geodesic $\gamma_{0,\infty}$, namely $X$ is the image of $S_0$ by a diagonal matrix of $\mathrm{SL}_2(\mathbb{R})$. As we have seen  the minimum is achieved uniquely at $X_0$. Finally by Proposition~\ref{Double2m+1gon}, the supremum is achieved by pairs of curves that are (images of) pairs of sides of $X_0$.
\end{proof}
\subsection{Interpolation between the regular n-gon and the staircase model}
Recall that Proposition~\ref{prop:other:def} provides another expression of $\KVol$:
\begin{equation}\label{sup_sinus}
K(X) = \sup_{(d,d')} K(d,d')\sin \theta(X,d,d'),
\end{equation}
where the supremum is taken over all pairs $(d,d')$ of distinct periodic directions. The quantity $K(d,d')$ is invariant under the diagonal action of the Veech group. 
Moreover, we know that Equation~\eqref{Expression_KVol} holds:
\begin{itemize}
\item for $X$ in the geodesic $\gamma_{0,\infty}$ by Proposition \ref{Examples_K}
\item for $X=X_0$ the double $(2m+1)$-gon by Corollary \ref{Victoire_en _X0}.
\end{itemize}

The main idea of the proof of Theorem~\ref{main_result} is to use these two results and interpolate between them to show that Equation~\eqref{Expression_KVol} holds in fact for the whole Teichmüller disc. By symmetry, we can restrict to the surfaces $X$ on the right half of the fundamental domain, that is on $\mathcal{D}_+ = \{x+iy \text{ | } 0\leq x \leq \frac{\Phi}{2} \text{ and } x^2+y^2 \geq 1\}$.
Using Equation~\eqref{sup_sinus}, it suffices to show that for any pair of distinct periodic directions $(d,d')$ one has:
\begin{equation}\label{etoile}
\tag{\ding{168}}
\forall X \in \mathcal{D}_+,\  K(d,d')\sin \theta(X,d,d') \leq  K(0,\infty)\sin \theta(X,0,\infty)
\end{equation}

The proof is divided in two steps:
\begin{enumerate}
\item Show that it suffices to prove~\eqref{etoile} for $0 \leq d < \frac{\Phi}{2} < d'$.
\item Show that~\eqref{etoile} holds under the assumption $0 \leq d < \frac{\Phi}{2} < d'$.
\end{enumerate}
The proof of the first step (Section~\ref{sec:reduction}) involves hyperbolic geometry and Veech group action, while the second step (Section~\ref{sec:ratio}) will be deduced from the study of the function 
$$X \mapsto \frac{\sin \theta(X,0,\infty)}{\sin \theta(X,d,d')}.$$

\subsection{Reduction to convenient geodesics}
\label{sec:reduction}
In this section, we prove that it suffices to verify~\eqref{etoile} for pairs $(d,d')$ with $0 \leq d < \frac{\Phi}{2} < d'$. The main step of the proof is Lemma \ref{Lemme_configurations}.

\begin{Def}
Given a pair of distinct periodic directions $(d,d')$ and its associated geodesic $\gamma_{d,d'}$ on $\HH$, we denote by $V(d,d')$ be the connected component of $\HH \backslash (\Gamma_n .\gamma_{d,d'})$ containing $X_0$.
\end{Def}

\begin{Rema}
If one of the images of $\gamma_{d,d'}$ by the action of the Veech group $\Gamma_n$ passes through $X_0$, then $V(d,d')$ is not well defined. It is convenient, in this case, to set $V(d,d') =\left\lbrace X_0 \right\rbrace $.  Note that  in this case there exists $G \in \Gamma_n$ such that 
$\sin \theta(X_0,G.d,G.d')=1$, so by Equation \eqref{etoile} for $X=X_0$  we have 
\begin{multline*}
K(d,d')\sin \theta(X,d,d') \leq K(d,d')=K(G.d,G.d')\sin \theta(X_0,G.d,G.d') \leq \\ 
\leq K(0,\infty) \sin \theta(X_0,0,\infty) \leq K(0,\infty) \sin \theta(X,0,\infty),
\end{multline*}
therefore, \eqref{etoile} holds for any $X \in \mathcal{D}_+$.
\end{Rema}

\begin{Rema}
Notice that, by definition, the boundary of $V(d,d')$ is made of geodesic segments in the Veech group orbit of $\gamma_{d,d'}$. Since $d$ and $d'$ correspond to directions of cusps, there is a finite number of such segments.
\end{Rema}

\begin{Lem}\label{Lemme_furthest_away}
For any pair of distinct periodic directions $(d,d')$, the furthest point $X_1$ from $\gamma_{0,\infty}$ in the boundary of $V(d,d') \cap \mathcal{D}_+$ is further away from $\gamma_{0,\infty}$ than any point $X\in \mathcal{D}_+$ outside $V(d,d')$. Equivalently
$$\sin \theta(X_1,0,\infty)  \leq \sin \theta(X,0,\infty). $$
\end{Lem}
\begin{proof}
Take $X \in \mathcal{D}_+ \setminus V(d,d')$ and call $g$ the perpendicular to $\gamma_{0,\infty}$ which contains $X$. If $g$ intersects 
$V(d,d')$ then there is an intersection point in the boundary of  $V(d,d')$, and this point is further away from $\gamma_{0,\infty}$ than $X$
(see Figure \ref{lemme_plus_loin_Voronoi}).
\begin{figure}
	 	\begin{center}
	 \definecolor{rvwvcq}{rgb}{0.08235294117647059,0.396078431372549,0.7529411764705882}

	 \definecolor{wrwrwr}{rgb}{0.3803921568627451,0.3803921568627451,0.3803921568627451}

	 \begin{tikzpicture}[scale=3,line cap=round,line join=round,>=triangle 45,x=1cm,y=1cm]

	 \clip(-1.5,-0.209218106995887) rectangle (2.5,1.974156378600822);

	 \draw [line width=2pt,color=wrwrwr] (0,0) -- (0,1.974156378600822);
	\draw [line width=1pt,color=wrwrwr] (-0.3,0)-- (1,0);

	 \draw [shift={(0,0)},line width=2pt,color=wrwrwr]  plot[domain=0.6283185307179586:1.5707963267948966,variable=\t]({1*1*cos(\t r)+0*1*sin(\t r)},{0*1*cos(\t r)+1*1*sin(\t r)});

	 \draw [line width=2pt,color=wrwrwr] (0.8090169943749475,0.5877852522924731) -- (0.8090169943749475,1.974156378600822);

	 \draw [shift={(0,0)},line width=2pt,color=wrwrwr]  plot[domain=1.0788881528646406:1.5707963267948966,variable=\t]({1*1.7128988675461305*cos(\t r)+0*1.7128988675461305*sin(\t r)},{0*1.7128988675461305*cos(\t r)+1*1.7128988675461305*sin(\t r)});

	 \draw [shift={(0.8090169943749475,0.5877852522924731)},line width=2pt,color=wrwrwr]  plot[domain=1.5707963267948966:2.509919796697609,variable=\t]({1*0.6116502162802975*cos(\t r)+0*0.6116502162802975*sin(\t r)},{0*0.6116502162802975*cos(\t r)+1*0.6116502162802975*sin(\t r)});

	 \draw [fill=rvwvcq] (0.45917695473250913,1.6502057613168712) circle (1pt);

	 \draw[color=rvwvcq] (0.4802469135802456,1.8) node {$X$};

	 \draw [fill=rvwvcq] (0,0) circle (1pt);

	 \draw[color=rvwvcq] (0.17,0.085473251028804516) node {$(0,0)$};

	 \draw [fill=rvwvcq] (0.8090169943749475,0.5877852522924731) circle (1pt);

	 \draw[color=rvwvcq] (0.92,0.6401646090534963) node {$X_0$};

	 \draw [fill=wrwrwr] (0.8090169943749475,1.1994354685727706) circle (1pt);

	 \draw[color=wrwrwr] (0.9,1.2511934156378586) node {$K$};

	 \draw [fill=wrwrwr] (0.8090169943749475,1.5098059588083967) circle (1pt);

	 \draw[color=wrwrwr] (.9,1.561975308641974) node {$M$};

	 \draw[color=wrwrwr] (1.12,0.95)
	  node {$\scriptstyle V(d,d') \cap \mathcal{D}_+$};

	 \end{tikzpicture}
\caption{The highest $K$ point of 
	$V(d,d') \cap \gamma_{\Phi / 2,\infty}$ is further away from $\gamma_{0,\infty}$ than $M$, which is itself further away from $\gamma_{0,\infty}$ than $X$.} \label{lemme_plus_loin_Voronoi}
\end{center}
\end{figure}
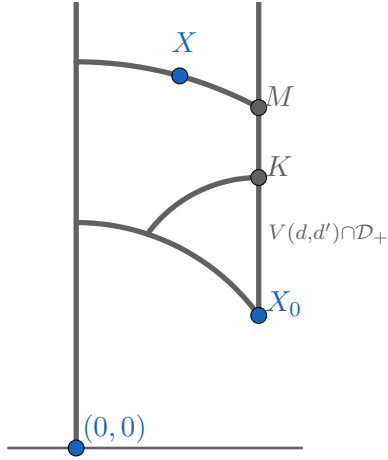
If $g$ does not meet $V(d,d')$, then $g$ intersects the geodesic $\gamma_{\Phi/2,\infty}$ above $V(d,d')$ at $M$. We claim that the highest point $K$ of $$V(d,d') \cap \gamma_{\Phi /2,\infty}$$ is further away from $\gamma_{0,\infty}$ than $M$. Indeed by construction we have the inequality 
$$\sin \theta(M,0,\infty) \geq \sin \theta(K,0,\infty).$$
By Proposition~\ref{prop:banane}, one has $\cosh d_{\mathrm{hyp}}(M,\gamma_{0,\infty}) = \sin^{-1} \theta (M,0,\infty)$. Since $\cosh$ is an increasing function, we deduce
$d_{\mathrm{hyp}}(M,\gamma_{0,\infty}) \leq d_{\mathrm{hyp}}(K,\gamma_{0,\infty})$. Now since $M$ is by construction further away from $\gamma_{0,\infty}$ than $X$, this proves the lemma.
\end{proof}

\begin{Lem}\label{Lemme_configurations}
Let $(d_1,d_1')$ and $(d_2,d_2')$ be two pairs of directions such that the associated geodesics $\gamma_{d_1,d_1'}$ and $\gamma_{d_2,d_2'}$ cross the half fundamental domain $\mathcal{D}_+$. We assume that:
\begin{enumerate}[label=(\roman*)]
\item $K(d_1,d_1') \geq K(d_2,d_2')$.
\item The geodesic $\gamma_{d_2,d_2'}$ lies outside $V(d_1,d_1')$. 
\item \eqref{etoile} holds for any pair of directions whose associated geodesic is in the boundary of $V(d_1,d_1')$.
\end{enumerate}
Then~\eqref{etoile} holds for $(d_2,d_2')$.
\end{Lem}
\begin{proof}
	Pick a point $X \in \mathcal{D}_+$. We subdivide the proof in two cases.\\	
\textbf{First case:} $X \not \in V(d_1,d_1')$. Then, Lemma \ref{Lemme_furthest_away} gives us a point $X_1$ on the boundary of $V(d_1,d_1')$ such that 
\begin{equation}
\label{eq:def:X1}
\sin \theta(X_1,0,\infty)  \leq \sin \theta(X,0,\infty).
\end{equation}
	Note that up to acting by the Veech group, which does not change the conclusion, we may assume that $X_1$ lies on $\gamma_{d_1,d_1'}$ itself, so $\sin \theta(X_1,d_1,d_1')=1$.
	Then 
	\begin{eqnarray*}
	K(d_2,d_2')\sin \theta(X,d_2,d_2') & \leq & K(d_2,d_2') \\
	& \leq & K(d_1,d_1') \text{ by assumption (i) }\\
	& = & K(d_1,d_1')\sin \theta(X_1,d_1,d_1')  \text{ because } \sin \theta(X_1,d_1,d_1') =1\\
	& \leq &  K(0,\infty)\cdot \sin \theta (X_1,0,\infty) \text{ by assumption (iii) }\\
		& \leq &  K(0,\infty)\cdot \sin \theta (X,0,\infty) \text{ by }~\eqref{eq:def:X1}.
	\end{eqnarray*}
		\textbf{Second case:} $X\in V(d_1,d_1')$. Then, by assumption (ii), the perpendicular  to $\gamma_{d_2,d_2'}$ through $X$ crosses the boundary of $V(d_1,d_1')$ before it reaches $\gamma_{d_2,d_2'}$, and again, up to acting by the Veech group we may assume the crossing occurs at $\gamma_{d_1,d_1'}$ so that $\sin \theta(X,d_1,d_1') \geq \sin \theta(X,d_2,d_2')$. Therefore
	\begin{eqnarray*}
	K(d_2,d_2')\sin \theta(X,d_2,d_2') & \leq & K(d_1,d_1')\sin \theta(X,d_2,d_2') \text{ by assumption (i) } \\
	& \leq & K(d_1,d_1')\sin \theta(X,d_1,d_1') \\
	& \leq &  K(0,\infty)\cdot \sin \theta (X,0,\infty) \text{ by assumption (iii),}
	\end{eqnarray*}
which finishes the proof.
\end{proof}

In particular, since we can apply Lemma \ref{Lemme_configurations} when $(d_2,d_2')$ is in the orbit of $(d_1,d_1')$ under the diagonal action of the Veech group, it suffices to prove~\eqref{etoile} for pairs of directions $(d,d')$ such that some segment of $\gamma_{d,d'}$ is in the boundary of $V(d_1, d'_1)$. Since $V(d_1, d'_1)$ is invariant under the dihedral group preserving $X_0$, it suffices to consider segments of the boundary which are contained in $\mathcal{D}_+$. These geodesics satisfy the following property.
\begin{Lem}\label{hypothese_dd'}
If $\gamma_{d,d'}$ is a geodesic whose closest point to $X_0$ lies  in  $\mathcal{D}_+$, then $\gamma_{d,d'}$  intersects  the geodesic $\gamma_{\Phi / 2, \infty}$. In particular, we can assume $d < \frac{\Phi}{2} < d'$. Moreover, the direction of the tangent vector of $\gamma_{d,d'}$ at the intersection point lies in the first quadrant, in particular $d+d'>\Phi$.
\end{Lem}
\begin{proof}
If $\gamma_{d,d'}$ does not intersect  the geodesic $\gamma_{\Phi / 2, \infty}$, then the perpendicular projection of $X_0$ to $\gamma_{d,d'}$ lies below $X_0$, hence not in $\mathcal{D}$, see Figure~\ref{lemme_tangent_vector_intersection_bord_D} (left part). The statement about the tangent vector at the intersection follows from the convexity of $V(d, d')$ and its symmetry with respect to $\gamma_{\Phi / 2, \infty}$. See Figure~\ref{lemme_tangent_vector_intersection_bord_D} (right part).
\end{proof}
 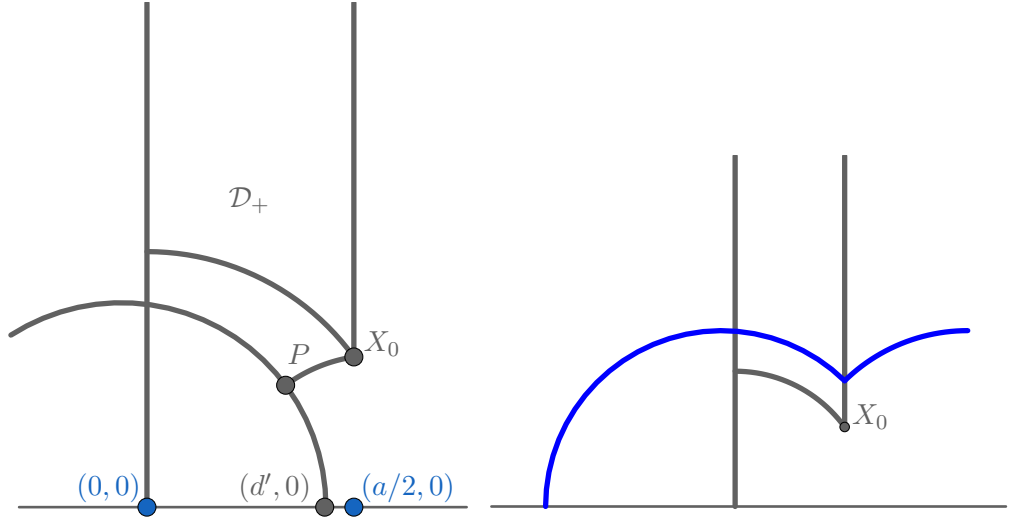
\begin{figure}[h]
 \hskip -55mm
 	\begin{minipage}[b]{0.1\linewidth}
		\begin{center}
\definecolor{rvwvcq}{rgb}{0.08235294117647059,0.396078431372549,0.7529411764705882}

\definecolor{wrwrwr}{rgb}{0.3803921568627451,0.3803921568627451,0.3803921568627451}

\begin{tikzpicture}[line cap=round,line join=round,>=triangle 45,x=1cm,y=1cm, scale=3.4]

\clip(-0.7,-0.209218106995887) rectangle (1.5,1.974156378600822);

\draw [line width=2pt,color=wrwrwr] (0,0) -- (0,1.974156378600822);

\draw [shift={(0,0)},line width=2pt,color=wrwrwr]  plot[domain=0.6283185307179586:1.5707963267948966,variable=\t]({1*1*cos(\t r)+0*1*sin(\t r)},{0*1*cos(\t r)+1*1*sin(\t r)});

\draw [line width=2pt,color=wrwrwr] (0.8090169943749475,0.5877852522924731) -- (0.8090169943749475,1.974156378600822);
\draw [line width=1pt,color=wrwrwr] (-0.5,0)-- (1.25,0);

\draw [shift={(-0.1,0)},line width=2pt,color=wrwrwr]  plot[domain=0:2.141592653589793,variable=\t]({1*0.8*cos(\t r)+0*0.8*sin(\t r)},{0*0.8*cos(\t r)+1*0.8*sin(\t r)});

\draw [shift={(0.8965728969241166,0)},line width=2pt,color=wrwrwr]  plot[domain=1.7186680249362207:2.209698132335832,variable=\t]({1*0.5942705939921021*cos(\t r)+0*0.5942705939921021*sin(\t r)},{0*0.5942705939921021*cos(\t r)+1*0.5942705939921021*sin(\t r)});

\draw [fill=rvwvcq] (0,0) circle (1pt);

\draw[color=rvwvcq] (-0.15,0.075473251028804516) node {$(0,0)$};

\draw [fill=wrwrwr] (0.8090169943749475,0.5877852522924731) circle (1pt);

\draw[color=wrwrwr] (0.92,0.6401646090534963) node {$X_0$};

\draw [fill=rvwvcq] (0.8090169943749475,0) circle (1pt);

\draw[color=rvwvcq] (1.02,0.075473251028804516) node {$(a/2,0)$};

\draw [fill=wrwrwr] (0.542200888640796,0.477051379443527) circle (1pt);

\draw[color=wrwrwr] (0.5942705939921021,0.6) node {$P$};

\draw [fill=wrwrwr] (0.695,0) circle (1pt);

\draw[color=wrwrwr] (0.5,0.0754) node {$(d',0)$};

\draw[color=wrwrwr] (0.4,1.2) node {$\mathcal{D}_+$};

\end{tikzpicture}
  \end{center}
\end{minipage}
\hskip 50mm
\begin{minipage}[b]{0.1\linewidth}

  \definecolor{qqqqff}{rgb}{0,0,1}

  \definecolor{wrwrwr}{rgb}{0.3803921568627451,0.3803921568627451,0.3803921568627451}

  \begin{tikzpicture}[line cap=round,line join=round,>=triangle 45,x=1cm,y=1cm, scale=1.8]
  \clip(-2.0838338230639577,-0.4) rectangle (3.064849755654652,2.591909856719528);

  \draw [line width=2pt,color=wrwrwr] (0,0) -- (0,2.591909856719528);
\draw [line width=1pt,color=wrwrwr] (-1.8,0)-- (2,0);

  \draw [shift={(0,0)},line width=2pt,color=wrwrwr]  plot[domain=0.6283185307179586:1.5707963267948966,variable=\t]({1*1*cos(\t r)+0*1*sin(\t r)},{0*1*cos(\t r)+1*1*sin(\t r)});

  \draw [line width=2pt,color=wrwrwr] (0.8090169943749475,0.5877852522924731) -- (0.8090169943749475,2.591909856719528);

  \draw [shift={(-0.1,0)},line width=2pt,color=qqqqff]  plot[domain=0.7964571147187204:3.141592653589793,variable=\t]({1*1.3*cos(\t r)+0*1.3*sin(\t r)},{0*1.3*cos(\t r)+1*1.3*sin(\t r)});

  \draw [shift={(1.718033988749895,0)},line width=2pt,color=qqqqff]  plot[domain=1.5707963267948966:2.345135538871073,variable=\t]({1*1.3*cos(\t r)+0*1.3*sin(\t r)},{0*1.3*cos(\t r)+1*1.3*sin(\t r)});

  \draw [fill=wrwrwr] (0.8090169943749475,0.5877852522924731) circle (1pt);

  \draw[color=wrwrwr] (1,0.6636155860364631) node {$X_0$};
  \end{tikzpicture}
  \end{minipage}
 \caption{Left: when $d'< \Phi/2$, the orthogonal projection of $X_0$ to $\gamma_{d,d'}$ does not lie in $\mathcal{D}$. Right: when the tangent vector to $\gamma_{d,d'}$ at the intersection with the right boundary of $\mathcal{D}$ does not lie in the first quadrant, $V(X_0,d, d')$ is not convex.} \label{lemme_tangent_vector_intersection_bord_D}
\end{figure}

In particular, to prove Theorem \ref{main_result} it suffices to show that~\eqref{etoile} holds for pairs $(d,d')$ with $d<\frac{\Phi}{2}<d'$ and $d+d'>\Phi$. We distinguish two cases:
\begin{enumerate}
\item $d\geq 0$
\item $d <0$
\end{enumerate}
In fact, case 1 is more difficult and will be proven in the next section. However, case 2 can be directly deduced from case 1. Indeed, by Lemma \ref{hypothese_dd'}, if $d <0$ then $d' \geq \Phi-d > \Phi$. In particular, $\gamma_{d,d'}$ lies outside $V(0,\Phi)$, whose boundary is made of the geodesic segments that are images of $\gamma_{0,\Phi}$ by the rotation around $X_0$, in particular $\gamma_{0,\Phi} \cap \mathcal{D}$ is the only boundary of $V(0,\Phi)$ intersecting $\mathcal{D}_+$; see Figure \ref{Voronoi_0_Phi} for the double pentagon. Since $K(d,d')\leq K(0,\Phi)$ (by Proposition~\ref{Examples_K}) and the pair $(0,\Phi)$ satisfies~\eqref{etoile} by case 1, we conclude by Lemma \ref{Lemme_configurations} that $(d,d')$ satisfies~\eqref{etoile}. This shows that case 1 implies case 2.

\begin{figure}[h]
\center
\definecolor{xdxdff}{rgb}{0.49019607843137253,0.49019607843137253,1}
\definecolor{qqqqff}{rgb}{0,0,1}
\definecolor{ccqqqq}{rgb}{0.8,0,0}
\begin{tikzpicture}[line cap=round,line join=round,>=triangle 45,x=5cm,y=5cm, scale = 0.8]
\clip(-0.3292891017575395,-0.191732664963307) rectangle (1.9,1.7);
\draw [line width=1pt] (0,0) -- (0,1.9606494654983015);
\draw [line width=1pt] (0.8090169943749475,0) -- (0.8090169943749475,1.9606494654983015);
\draw [line width=1pt,color=ccqqqq] (0.6180339887498948,0) -- (0.6180339887498948,1.9606494654983015);
\draw [shift={(0.8090169943749475,0)},line width=1pt,color=ccqqqq]  plot[domain=0:3.141592653589793,variable=\t]({1*0.8090169943749475*cos(\t r)+0*0.8090169943749475*sin(\t r)},{0*0.8090169943749475*cos(\t r)+1*0.8090169943749475*sin(\t r)});
\draw [line width=1pt] (1.618033988749895,0) -- (1.618033988749895,1.9606494654983015);
\draw [line width=1pt,color=ccqqqq] (1,0) -- (1,1.9606494654983015);
\draw [shift={(0.8090169943749475,0)},line width=1pt,dash pattern=on 3pt off 3pt]  plot[domain=0:3.141592653589793,variable=\t]({1*0.19098300562505255*cos(\t r)+0*0.19098300562505255*sin(\t r)},{0*0.19098300562505255*cos(\t r)+1*0.19098300562505255*sin(\t r)});
\draw [shift={(0.6180339887498948,0)},line width=1pt,dash pattern=on 3pt off 3pt]  plot[domain=0:3.141592653589793,variable=\t]({1*0.6180339887498948*cos(\t r)+0*0.6180339887498948*sin(\t r)},{0*0.6180339887498948*cos(\t r)+1*0.6180339887498948*sin(\t r)});
\draw [shift={(1,0)},line width=1pt,dash pattern=on 3pt off 3pt]  plot[domain=0:3.141592653589793,variable=\t]({1*0.6180339887498949*cos(\t r)+0*0.6180339887498949*sin(\t r)},{0*0.6180339887498949*cos(\t r)+1*0.6180339887498949*sin(\t r)});
\draw [shift={(0.3090169943749474,0)},line width=1pt,dash pattern=on 3pt off 3pt]  plot[domain=0:3.141592653589793,variable=\t]({1*0.3090169943749474*cos(\t r)+0*0.3090169943749474*sin(\t r)},{0*0.3090169943749474*cos(\t r)+1*0.3090169943749474*sin(\t r)});
\draw [shift={(1.3090169943749475,0)},line width=1pt,dash pattern=on 3pt off 3pt]  plot[domain=0:3.141592653589793,variable=\t]({1*0.30901699437494745*cos(\t r)+0*0.30901699437494745*sin(\t r)},{0*0.30901699437494745*cos(\t r)+1*0.30901699437494745*sin(\t r)});
\draw [shift={(0.5,0)},line width=1pt,color=ccqqqq]  plot[domain=0:3.141592653589793,variable=\t]({1*0.5*cos(\t r)+0*0.5*sin(\t r)},{0*0.5*cos(\t r)+1*0.5*sin(\t r)});
\draw [shift={(1.118033988749895,0)},line width=1pt,color=ccqqqq]  plot[domain=0:3.141592653589793,variable=\t]({1*0.5*cos(\t r)+0*0.5*sin(\t r)},{0*0.5*cos(\t r)+1*0.5*sin(\t r)});
\draw [line width=1pt,domain=-0.7292891017575395:2.9610233706005773] plot(\x,{(-0-0*\x)/1.618033988749895});
\draw (-0.05259171217934269,0) node[anchor=north west] {$0$};
\draw (0.569465515507447,0) node[anchor=north west] {$\frac{1}{\Phi}$};
\draw (0.7544014480629789,0) node[anchor=north west] {$\frac{\Phi}{2}$};
\draw (0.9561497381235594,0) node[anchor=north west] {$1$};
\draw (1.5529884295527763,0) node[anchor=north west] {$\Phi$};
\draw (1.1789134750654502,0) node[anchor=north west] {$\frac{2}{\Phi}$};
\draw (0.2584369016640521,0) node[anchor=north west] {$\frac{\Phi^2-2}{\Phi}$};
\draw [shift={(0,0)},line width=1pt,dash pattern=on 3pt off 3pt]  plot[domain=0:0.6283185307179588,variable=\t]({1*1*cos(\t r)+0*1*sin(\t r)},{0*1*cos(\t r)+1*1*sin(\t r)});
\draw [shift={(0,0)},line width=1pt]  plot[domain=0.6283185307179587:1.5707963267948966,variable=\t]({1*1*cos(\t r)+0*1*sin(\t r)},{0*1*cos(\t r)+1*1*sin(\t r)});
\draw [shift={(1.618033988749895,0)},line width=1pt]  plot[domain=1.5707963267948966:2.5132741228718345,variable=\t]({1*1*cos(\t r)+0*1*sin(\t r)},{0*1*cos(\t r)+1*1*sin(\t r)});
\draw [shift={(1.618033988749895,0)},line width=1pt,dash pattern=on 3pt off 3pt]  plot[domain=2.513274122871835:3.141592653589793,variable=\t]({1*1*cos(\t r)+0*1*sin(\t r)},{0*1*cos(\t r)+1*1*sin(\t r)});
\begin{scriptsize}
\draw [fill=qqqqff] (0.8090169943749472,0.5877852522924731) circle (4pt);
\draw[color=qqqqff] (0.865,0.75) node {\large $X_{0}$};
\draw [fill=qqqqff] (0,0.9955555555555539) circle (3pt);
\draw[color=qqqqff] (0.06,1.06) node {\large $S_{0}$};
\end{scriptsize}
\end{tikzpicture}
\caption{The domain $V(0, \Phi)$ for the double pentagon.}
\label{Voronoi_0_Phi}
\end{figure}
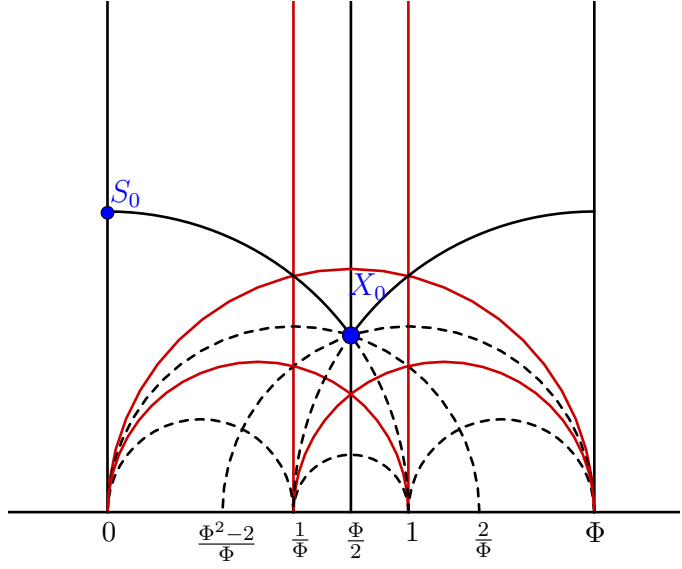

\subsection{Study of the ratio of sines}
\label{sec:ratio}
In this section we show that any pair of periodic directions $(d,d')$ in case 1 (i.e $0\leq d < \frac{\Phi}{2} < d'$) satisfies~\eqref{etoile}. Our proof relies on the study of the function 
\begin{equation*}
F_{(d,d')}: X \mapsto \frac{\sin \theta(X,0,\infty)}{\sin \theta(X,d,d')}.
\end{equation*}
More precisely:

\begin{Prop}\label{minimality}
Under the assumption $0\leq d < \frac{\Phi}{2} < d'$, the function $F_{(d,d')}$ on $\mathcal{D}_+$ is minimal at $X_0$.
\end{Prop}
Before giving the proof of this proposition, let us first state and prove the following corollary, which concludes the proof of Theorem \ref{main_result}:
\begin{Cor}\label{corfinal}
For any $(d,d')$ such that $0\leq d < \frac{\Phi}{2} < d'$,~\eqref{etoile} holds.
\end{Cor}
\begin{proof}[Proof of Corollary \ref{corfinal}]
Let $(d,d')$ be such that $0\leq d < \frac{\Phi}{2} < d'$, and $X \in \mathcal{D}_+$. We know from Corollary \ref{Victoire_en _X0} that
\begin{equation*}
K(d,d') \sin \theta(X_0,d,d') \leq K(0,\infty) \sin \theta(X_0, 0,\infty).
\end{equation*}
In particular, by minimality of $F_{(d,d')}$ at $X_0$
\begin{equation*}
K(d,d')\leq K(0,\infty) F_{(d,d')}(X_0) \leq K(0,\infty) F_{(d,d')}(X) 
\end{equation*}
Hence
\begin{equation*}
K(d,d') \sin \theta(X,d,d')\leq K(0,\infty) \sin \theta(X, 0,\infty)
\end{equation*}
This concludes the proof.
\end{proof}

\begin{proof}[Proof of Proposition \ref{minimality}] We divide the proof in 5 steps:
\begin{enumerate}
\item We remark that $F_{(d,d')}$ is well defined and differentiable on $\HH$, and has a well defined minimum on $\mathcal{D}_+$. 
\item We study the gradient of $F_{(d,d')}$ in $\mathcal{D}_+$ and show that it doesn't vanish inside~$\mathcal{D}_+$.

\item We remark that $F_{(d,d')}$ is not minimal at the left boundary of $\mathcal{D}_+$.

\item We study the variations of $F_{(d,d')}$ on the lower boundary of $\mathcal{D}_+$, which we  parametrize as 
$\left\lbrace (\cos \theta, \sin \theta) : \theta \in 
\left[ \frac{\pi}{n}, \frac{\pi}{2} \right] \right\rbrace $,  and show that $F_{(d,d')}$ increases with $\theta$.

\item We study the variations of $F_{(d,d')}$ on the line $x=\frac{\Phi}{2}$ and show that it increases strictly with $y$.
\end{enumerate}

\textbf{Proof of step 1.}\newline
Note that by Proposition~\ref{prop:banane}
 \[  
 F_{(d,d')}(X)=
 \frac{ \cosh \mbox{d}_{hyp} (X, \gamma_{d,d'})}{ \cosh \mbox{d}_{hyp} (X,  \gamma_{0,\infty})}
 \]
where $\mbox{d}_{hyp} (X, \gamma_{d,d'})$ is the hyperbolic distance from $X$ to the geodesic $\gamma_{d,d'}$. Distance functions are not differentiable, but their $\cosh$'s are. 

Moreover, $F_{(d,d')}(x+iy) \to +\infty$ when $y \to +\infty$ (and $x \in [0,\frac{\Phi}{2}]$), so if $A>0$ is sufficiently big, we have $F_{(d,d')}(X) > F_{(d,d')}(X_0)$ for any $X=x+iy \in \mathcal{D}_+$ with $y>A$. In particular, $F_{(d,d')}$ reaches its minimum on the compact set $\mathcal{K} = \mathcal{D}_+\cap \{x+iy| y\leq A\}$. This finishes the proof of step 1.\newline

\textbf{Proof of step 2.}\newline
Note that since the natural logarithm is an increasing diffeomorphism, we may as well look for the minimum of $\log F_{(d,d')}(X)$ over $\mathcal{D}_+$. Now 

\begin{eqnarray*}
\nabla \log F_{(d,d')}(X)&=& 
\frac{\nabla \cosh \mbox{d}_{hyp} (X, \gamma_{d,d'})}
{\cosh \mbox{d}_{hyp} (X, \gamma_{d,d'})} -
\frac{\nabla \cosh \mbox{d}_{hyp} (X,  \gamma_{0,\infty})}
{\cosh \mbox{d}_{hyp} (X,  \gamma_{0,\infty})}\\
&=& \tanh \mbox{d}_{hyp} (X, \gamma_{d,d'}) \nabla  \mbox{d}_{hyp} (X,  \gamma_{d,d'})-
\tanh \mbox{d}_{hyp} (X,  \gamma_{0,\infty}) \nabla \mbox{d}_{hyp} (X,  \gamma_{0,\infty}).
\end{eqnarray*}
Now the distance gradients are unit vectors, and they are parallel only along the common perpendicular (if it exists) to $\gamma_{0,\infty}$ and $\gamma_{d,d'}$, or never (otherwise), so the gradient of $F_{(d,d')}$ cannot vanish outside of the common perpendicular. Along the common perpendicular, the numbers 
$\tanh \mbox{d}_{hyp} (X, \gamma_{d,d'})$ and $\tanh \mbox{d}_{hyp} (X,  \gamma_{0,\infty})$ are equal only at the middle of the common perpendicular segment between the two lines, and there the distance gradients point in opposite directions. So the gradient of $F_{(d,d')}$ cannot vanish at all. Thus $F_{(d,d')}$ does not have a minimum in the interior of $\mathcal{D}_+$.\newline
		
\textbf{Proof of step 3.}\newline
On the left boundary of $\mathcal{D}_+$, which is contained in $\gamma_{0,\infty}$,
 we have $ \nabla  \mbox{d}_{hyp} (X,  \gamma_{0,\infty})=0$, and $ \nabla  \mbox{d}_{hyp} (X, \gamma_{d,d'})$ points to the left because $d \geq 0$. Therefore no point on the left boundary is a local minimum for $F_{(d,d')}$.\newline

\textbf{Proof of step 4.}\newline
Now let us study the function $F_{(d,d')}(X)$ restricted to the lower boundary of $\mathcal{D}_+$.

We first give a more convenient expression of $F_{(d,d')}$. Let $X = x+ iy$ be a point in the domain $\mathcal{D}_+$. We have:
\begin{equation*}
\sin \theta(X, 0,\infty)  = \frac{y}{\sqrt{x^2+y^2}}
\end{equation*}
And since the matrix $\begin{pmatrix} -1 & d \\ \frac{1}{d'-d} & \frac{-d'}{d'-d} \end{pmatrix} \in SL_2(\RR)$ acts on $\HH$ by isometry and sends the geodesic $\gamma_{d,d'}$ to $\gamma_{0,\infty}$ and $x+iy$ to 
\begin{equation*}
\tilde{x}+i\tilde{y} = \frac{(-x-iy+d)(d'-d)}{x+iy-d'} =\frac{d-d'}{(x-d')^2+y^2}\cdot (-(x-d)(x-d')-y^2+iy(d'-d))
\end{equation*}
we have:
\begin{equation*}
\sin \theta(X, d,d') = \frac{\tilde{y}}{\sqrt{\tilde{x}^2+\tilde{y}^2}} = \frac{y(d'-d)}{\sqrt{((x-d)(x-d')+y^2)^2+y^2(d'-d)^2}}
\end{equation*}
Hence:
\begin{equation}\label{F_en x et y}
F_{(d,d')}(X) = \frac{1}{d'-d}\sqrt{\frac{((x-d)(x-d')+y^2)^2+y^2(d'-d)^2}{x^2+y^2}}
\end{equation}
To study the  variations of $F_{(d,d')}(X)$, it suffices to consider what is inside the square root in  (\ref{F_en x et y}):
\begin{equation*}
G: (x,y)\mapsto \frac{((x-d)(x-d')+y^2)^2+y^2(d'-d)^2}{x^2+y^2}
\end{equation*}
On the lower boundary of $\mathcal{D}_+$, 
this reduces to 
\[
(1+dd' -(d+d') \cos \theta)^2 + (d'-d)^2 \sin^2 \theta=
(1+dd')^2 - 2(1+dd')(d+d')\cos \theta + 2dd' \cos 2\theta -1
\]
whose derivative with respect to $\theta$ is
\[ 
-4dd' \sin 2\theta +2(1+dd')(d+d')\sin \theta =
2 \sin \theta \left[(1+dd')(d+d')-4dd'\cos \theta \right]. 
\]
We want to prove that $G$ is an increasing function of $\theta$. This follows from $(1+dd')(d+d') \geq 4dd'$, which in turn follows from $d+d' \geq 2 \sqrt{dd'}$, and the fact that $x^2-2x+1 \geq0$ for any real number $x$, in particular for $x=\sqrt{dd'}$.\newline

\textbf{Proof of step 5.}\newline
We compute the differential of $G$ with respect to $y$. It gives
\begin{equation*}
\frac{\partial G}{\partial y}(x,y) = \frac{2y}{(x^2+y^2)^2} \cdot (y^4+2x^2y^2 + x^4 - dd'(2x-d)(2x-d')).
\end{equation*}
The sign of $\frac{\partial G}{\partial y}(x,y)$ is the sign of the polynomial $P(X) = X^2+2x^2X + x^4 - dd'(2x-d)(2x-d')$, which has discriminant
\begin{equation*}
\Delta = 4dd'(2x-d)(2x-d').
\end{equation*}
Setting $x = \frac{\Phi}{2}$ yields:
\begin{equation*}
\Delta = 4dd'(\Phi-d)(\Phi-d').
\end{equation*}
In particular:
\begin{itemize}
\item If $d'>\Phi$, then $\Delta <0$ and $P$ has no real roots.
\item If $d'=\Phi$ then $\Delta=0$ and the only real root of $P$ is $-\frac{\Phi^2}{4} <0$.
\item Else, $\frac{\Phi}{2}<d' < \Phi$ and $P$ has two real roots:
\begin{equation*}
\lambda_- = -\frac{\Phi^2}{4} - \sqrt{dd'(\Phi-d)(\Phi-d')} <0 \text{ and } \lambda_+ = -\frac{\Phi^2}{4} + \sqrt{dd'(\Phi-d)(\Phi-d')}.
\end{equation*}
But $d(\Phi-d) \leq \frac{\Phi^2}{4}$ and $d'(\Phi-d') \leq \frac{\Phi^2}{4}$ so $\sqrt{dd'(\Phi-d)(\Phi-d')} \leq \frac{\Phi^2}{4}$ and $\lambda_+ \leq 0$.
\end{itemize}
In conclusion, $P$ has no real positive roots, in particular it is positive on $\RR_+^*$, and so is $\frac{\partial G}{\partial y}(\frac{\Phi}{2},y)$. This finishes the proof of the last step.
\newline

In particular, the only possible minimum for $F_{(d,d')}$ on $\mathcal{D}_+$ is $X_0$. This proves Proposition \ref{minimality}.
\end{proof}

\bibliographystyle{alpha}
\bibliography{KVol_bibli}
\end{document}